\documentclass{amsart}

\usepackage[utf8]{inputenc}
\usepackage{amssymb}
\usepackage{mathrsfs}
\usepackage{todonotes}
\usepackage[foot]{amsaddr}
\usepackage[alphabetic, abbrev]{amsrefs}

\usepackage[bottom=3cm, top=3cm, left=3.5cm, right=3.5cm]{geometry}

\usepackage{xcolor}
\usepackage[bookmarks=true]{hyperref}
\hypersetup{
    colorlinks,
    linkcolor={red!50!black},
    citecolor={blue!50!black},
    urlcolor={blue!80!black},
}

\usepackage{bbm}
\usepackage{mathtools}
\usepackage[shortlabels]{enumitem}
\usepackage[capitalise]{cleveref}

\usepackage[normalem]{ulem}
\usepackage{cancel}

\theoremstyle{plain}
\newtheorem{theorem}{Theorem}
\newtheorem*{theorem*}{Theorem}
\newtheorem{proposition}[theorem]{Proposition}

\newtheorem{lemma}[theorem]{Lemma}
\newtheorem*{lemma*}{Lemma}

\newtheorem{thm}{Theorem}[section]

\newtheorem*{lem*}{Lemma}

\theoremstyle{definition}
\newtheorem{definition}[thm]{Definition}

\theoremstyle{remark}
\newtheorem{remark}[thm]{Remark}
\newtheorem{example}[thm]{Example}

\crefname{theorem}{Theorem}{Theorems}						
\crefname{theoremintro}{Theorem}{Theorems}				

\newcommand{\Tor}{\mathrm{Tor}}

\definecolor{ggreen}{rgb}{0.2, 0.57, 0.1}

\newcommand{\R}{\mathbb{R}}
\newcommand{\diff}{d}
\newcommand{\de}{\,{\rm d}}											
\newcommand{\p}{\partial}												

\newcommand{\g}{\gamma}
\newcommand{\eps}{\varepsilon}
\newcommand{\la}{\langle}
\newcommand{\ra}{\rangle}

\newcommand{\sr}{sub-{R}ie\-man\-nian }				
\newcommand{\distr}{\mathcal{D}}								

\newcommand{\B}{\beta}

\newcommand{\lift}[1]{\overline{#1}} 							
\newcommand{\mg}{\B_1\nu_1\wedge \omega
										+\B_2\nu_2\wedge \omega}	

\newcommand\restr[2]{{													
  \left.\kern-\nulldelimiterspace 					
  #1 																			
  \right|_{#2} 														
  }}

\title{ magnetic flows on 3D contact Sub-Riemannian\\ manifolds
via the Rumin Complex 
  }
\date{\today}

\author{Davide Barilari}
\email{\href{mailto:davide.barilari@unipd.it}{davide.barilari@unipd.it}}

\author{Tania Bossio}
\address{University of Fribourg, Chemin du Musée 23, 1700 Fribourg, Switzerland}
\address{Dipartimento di Matematica ``Tullio Levi-Civita'', Universit\`a degli Studi di Padova}
\email{\href{mailto:tania.bossio@unipd.it}{tania.bossio@unipd.it}}

\author{Valentina Franceschi}
\email{\href{mailto:valentina.franceschi@unipd.it}{valentina.franceschi@unipd.it}}

\begin{document}

\begin{abstract}

 We show that the appropriate  notion of magnetic field on three-dimensional contact sub-Riemannian manifolds is given by a closed Rumin differential two-form. 
We introduce horizontal magnetic flows starting from magnetic potential one-forms, proving that the flow depends only on the Rumin differential of the potential. Notably, in dimension three the Rumin differential acts on one-forms as a second-order differential operator.


We further prove that such magnetic flows can be interpreted as a geodesic flow on a suitably lifted sub-Riemannian structure, which is of Engel type when the magnetic field is non-vanishing.
In the general case, when the magnetic field may vanish, we analyze the geometry of the lifted structure, characterizing its step and abnormal trajectories in terms of the analytical properties of the magnetic field.

Our work is inspired by the classical correspondence, first observed by Montgomery in \cites{Mon90,MonZero}, between Riemannian magnetic flows and sub-Riemannian geometry.

\end{abstract}

\maketitle
\tableofcontents

\section{Introduction}
\label{s:intro}

The correspondence between magnetic flows on Riemannian manifolds and sub-Riemannian geometry is nowadays well established. It was clearly demonstrated by Montgomery in his seminal work \emph{``Hearing the zero locus of a magnetic field''} \cite{MonZero}, where he studies the asymptotics of a two-dimensional quantum particle in a magnetic field and reinterprets its quantum dynamics in terms of a suitable sub-Riemannian Laplacian on a lifted space (see also \cite[Chapter~12]{MonBook} and the earlier works \cites{Mon90,Mon94}). Inspired by these developments, in this paper we focus on the classical dynamics perspective, initiating a systematic study of magnetic flows on sub-Riemannian manifolds. 

The simplest instance of the correspondence just mentioned permits to reinterpret the geodesic flow on the sub-Riemannian Heisenberg group as the lift of the magnetic flow of a charged particle moving in the Euclidean plane under the influence of a constant magnetic field -- a construction we now recall (cf.\ \cite[Section~4.4.2]{thebook}).

\smallskip
The Heisenberg group is nowadays recognized as the prototype model of \sr geometry; i.e., spaces where the distance is computed minimizing the length of curves that are tangent to a suitable non-integrable vector distribution.  In the specific case of the Heisenberg group, which is topologically $\R^{3}$ with coordinates $(x,y,z)$, the vector distribution $\distr$ is given by
 \begin{equation} \label{eq:intro0}
	\distr=\ker \omega,\qquad \omega=dz-\frac{1}{2}(xdy-ydx). 
 \end{equation}
The length of a curve tangent to $\distr$ (also said \emph{horizontal}) is set to be equal to the Euclidean length of its projection onto the $(x,y)$ plane. Length-minimizers among horizontal curves joining two fixed points always exist: these are curves in $\R^{3}$ whose projection onto the $(x,y)$ plane are circles or lines and whose vertical coordinate is recovered by imposing the relation \eqref{eq:intro0}, namely
 \begin{equation} \label{eq:intro00}
	 z(t)-z(0)=\frac{1}{2}\int_{0}^{t} \big( x(\tau)\dot y(\tau)-y(\tau)\dot x(\tau) \big) \, d\tau. 
 \end{equation}

It is possible to reinterpret such geodesic flow on the \sr Heisenberg group as the lift of the motion of a charged particle confined to the \((x,y)\) plane, subject to a constant vector magnetic field \(\vec{B}\) oriented positively along the \(z\)-axis.  
Indeed, according to the classical Lorentz force law,  
$F = q (\mathbf{v} \times \vec{B})$,
such a particle moves with an acceleration in the \((x,y)\)  plane which is perpendicular to the magnetic field direction. The particle then follows a circular trajectory in the plane with curvature proportional to the charge $q$; if the particle is not charged then the trajectory is a straight line in the direction of $\mathbf{v}$.

The relation \eqref{eq:intro00} can be interpreted in magnetic terms: its right-hand side represents the integral of the differential 1-form  
$A = \frac{1}{2} (x dy - y dx),$
treated as a magnetic potential, over the planar curve. The exterior derivative of the magnetic potential,  
$\beta = dA = dx \wedge dy,$
corresponds to the magnetic field, viewed as a differential 2-form. In fact, the magnetic 2-form \(\beta\) is closed and is dual to the vector field \(\vec{B}\) under the standard isomorphism between 2-forms and vector fields in \(\mathbb{R}^{3}\).

\smallskip
More generally a \emph{magnetic field} on a two-dimensional Riemannian surface $(N,g)$ is represented by a differential 2-form $\beta\in \Omega^{2}(N)$ which is closed, i.e., satisfies $d\beta=0$.
Assuming that $\beta=dA$ for some \emph{magnetic potential} $A\in \Omega^{1}(N)$, then the motion of a particle with charge $q\in \R$ on $N$ under the action of the corresponding magnetic field is obtained by minimizing the action associated with the following Lagrangian 
\begin{equation}
\label{eq:lagrintro}
L(p,v)=\frac{1}{2}\|v\|^{2}_{g}-  q\langle A(p), v\rangle,\qquad p\in N, v\in T_{p}N.
\end{equation}
 where the brackets represent the duality between covectors and vectors. The corresponding Hamiltonian, denoting $g^{*}$ the cometric on $T^{*}N$,  is given by
 \begin{equation}
\label{eq:Ham0}
H(p,\xi)=\frac{1}{2}\| \xi+qA(p)\|^{2}_{g^{*}},\qquad p\in N,\, \xi \in T^{*}_{p}N.
\end{equation}
Magnetic geodesics with charge $q$ satisfy the following equation with respect to the Levi-Civita connection $\nabla$
\begin{equation}\label{eq:Omega}
\nabla_{\dot \gamma}\dot \gamma= q\, b(\dot \gamma),
\end{equation}
where $b$ is the skew-symmetric operator defined by the identity $\beta(v,w)=g(v,b(w))$. 
Since $b$ is skew-symmetric and $N$ is 2-dimensional, then $b(\dot \gamma)$ is a scalar multiple of $\dot \gamma^{\perp}$, the unit vector orthogonal to $\dot \gamma$ in $N$. Then the relation \eqref{eq:Omega} can also be rewritten in scalar terms:  taking the inner product in both sides of \eqref{eq:Omega} with $\dot \gamma^{\perp}$ and using that $k_{g}(\gamma)=g(\nabla_{\dot \gamma}\dot \gamma,\dot \gamma^{\perp})$ is the (signed) geodesic curvature of $\gamma$ we have
\begin{equation}\label{eq:Omega2}
k_{g}(\gamma)= q\, b(\dot \gamma),
\end{equation}
where now the right-hand side is treated as a scalar.  Recall that the  (unsigned) geodesic curvature appears as a metric invariant satisfying (here $d_{g}$ is the Riemannian distance on $N$)
 \begin{equation}\label{eq:db16}
d_{g}^{2}(\gamma(t+\eps),\gamma(t))=\eps^{2}-\frac{k^{2}_{g}(\gamma)}{12}\eps^{4}+o(\eps^{4}).
\end{equation}

 Similarly to the construction described above, the magnetic motion on $N$ can be lifted to a motion in the extended space $\overline N=N\times \R$, imposing that the curve is tangent to the distribution
 \begin{equation} \label{eq:intro1}
	\distr=\ker \omega,\qquad \omega=dz-\pi^{*}A.
 \end{equation}
Here $\pi:\overline N \to N$ denotes the canonical projection, and the length of a curve tangent to $\distr$ is set to be equal to the Riemannian length of its projection onto $N$.
The distribution $\distr$ has rank two, and, by definition, it is a contact distribution at points where 
 the differential 
$d\omega|_{\distr}$ 
 is non-degenerate. 
Due to $d\omega=-\pi^{*}\beta$, this is in fact equivalent for the magnetic field $\B$ to be non-vanishing on $N$. 
One can prove (cf.\ for instance \cite{AG99}, see also \cite[Sec.\ 4.4]{thebook}) that, under this assumption, magnetic trajectories on $N$ with charge $q$ coincide with the projection of normal extremals for the \sr problem  which has $q$ as vertical part of the covector.

 Let $\mathcal{Z}\subset N$ be the set of points where the magnetic field $\B$ vanishes. The lifted \sr structure has step higher than 2 at $\mathcal{Z}\times \R\subset \lift N$,
 the lift of the zero locus that is also called \emph{Martinet set}. 
Moreover, any horizontal curve contained in the Martinet set is an abnormal extremal for the \sr structure, see~\cite[Lemma~4.43]{thebook}. 

For a more explicit construction of the lifted \sr structure, let us consider a local orthonormal frame $X_{1},X_{2}$ for the Riemannian metric $g$ on $N$. Then the lifted distribution \eqref{eq:intro1} in $\lift N$ is spanned by the vector fields
\begin{equation}
Y_{1}=X_{1}+ A(X_{1})\partial_z,\qquad Y_{2}=X_{2}+ A(X_{2})\partial_z,
\end{equation}
where now we identify $T\lift N$ with $TN\times \R$ and we denote by $z$ the new variable. By construction, the vector field $[X_{1},X_{2}]$ is tangent to $N$ so let us write $[X_{1},X_{2}]=c_{1}X_{1}+c_{2}X_{2}$ for suitable smooth functions $c_{1},c_{2}\in C^{\infty}(N)$. It is easy to see using  $dA=\beta$ that
 $$[Y_{1},Y_{2}]=c_{1}Y_{1}+c_{2}Y_{2}+\beta(X_{1},X_{2})\partial_{z}.$$
hence $[Y_{1},Y_{2}]$ is linearly independent from $\distr$ at points where the magnetic field $\beta$ does not vanish, and $\distr$ has step 2. More in general, one can prove that the distribution is bracket generating of step $k+2$ if $k\in \mathbb N$ is the smallest integer such that there exists a derivative $\partial^{\alpha}\beta$ which is non-vanishing, with $\alpha$ a multi-index satisfying $|\alpha|=k$ (cf. \cite[Exercise~4.44]{thebook}). We refer to \cite[Chapter~12]{MonBook} for a general discussion of magnetic fields on higher dimensional Riemannian manifolds (see also \cite{LZ11} for the relation with the corresponding sub-Riemannian structure).

\subsection{Magnetic fields on \sr manifolds}
The goal of this paper is to investigate magnetic fields on \sr manifolds $M$. More specifically we are interested into 
\begin{itemize}
\item[(i)] formalizing what is a magnetic field on a \sr manifold $M$, 
\item[(ii)] understanding magnetic geodesics on $M$, 
\item[(iii)] building a suitable lifted sub-Riemannian structure $\lift M$,
\item[(iv)] interpreting geometric properties of the lifted structures in terms of the magnetic field. 
\end{itemize}
In this paper we focus on the lowest dimensional case, namely when $M$ is a three-dimensional contact \sr manifold.

To generalize the approach described above, one might start by considering a closed differential 2-form $\beta\in \Omega^{2}(M)$ on the \sr manifold representing the magnetic field, which again, we assume exact, i.e., $\beta=dA$ for some \emph{magnetic potential} $A\in \Omega^{1}(M)$. 

However, interpreting heuristically \eqref{eq:lagrintro} in the \sr sense, namely when the vectors $v$ belong to the distribution $\distr=\ker\omega$, being $\omega$ a contact form, one immediately observes that adding to $A$ a term which is proportional to $\omega$ one does not affect the Lagrangian $L$, i.e., we can replace $A$ by any other $A'=A+g\omega$, where $g\in C^{\infty}(M)$. 

Similarly, since we are interested in minimizing the action defined by $L$, which we recall is defined by the integral of the Lagrangian \eqref{eq:lagrintro} over horizontal paths, we are free to add to $A$ an exact term, i.e., replace $A$ by any $A'=A+df$, where $f$ is a smooth function on $M$. Actually, due to the invariance of $A$ with respect to the contact form, we can reduce to add the horizontal part of the differential $df$.  See Section~\ref{prel:Rumin} below.

These two properties can be formalized by saying that a magnetic field in the \sr framework is represented by a \emph{Rumin differential 2-form} $\beta\in \Omega^{2}_{H}(M)$ which is closed.

\subsection{The Rumin complex and magnetic geodesic equations}
The Rumin complex is a refinement of the de Rham complex designed to better capture the geometry of contact (or, more in general, sub-Rie\-man\-nian) manifolds. Originally, it was introduced by Rumin in the framework of contact manifolds \cites{RuminContact}, and later extended to Carnot groups \cites{Rum00a}.  

For a three-dimensional contact manifold, with $\distr=\operatorname{span}\{X_1,X_2\}$ and Reeb vector field $X_0$, the Rumin complex is defined by the space of  \emph{horizontal $k$-forms} $\Omega^k_H(M)$ as follows:
\begin{gather}
\Omega^0_H(M):=C^\infty(M),\qquad 
\Omega^1_H(M):=\mathrm{span}\{\nu_1,\nu_2\},\\ 
\Omega^2_H(M):=\mathrm{span}\{\nu_1\wedge\omega,\nu_2\wedge\omega\}, \qquad
\Omega^3_H(M):=\mathrm{span}\{\nu_1\wedge\nu_2 \wedge \omega\},
\end{gather}
where $\nu_1,\nu_2,\omega$ is the dual basis to  $X_1,X_2,X_0$. The standard exterior derivative is replaced with the horizontal differential $\diff_H$, which is adapted to the filtration naturally induced by the contact structure. 
This new differential operator, when acting on 1-forms, has degree two. In our case, for $f\in C^{\infty}(M)$ and $A=A_1\nu_1+A_2\nu_2\in \Omega_H^1$,
it is explicitly given by \begin{equation}
		\diff^0_H f 
		=(X_1f)\nu_1+(X_2f)\nu_2;\qquad 
		\diff^1_HA=\diff\left(A+\diff A(X_1,X_2)\omega\right);
	\end{equation}
while $\diff^2_H=\diff$ coincides with the exterior differential.  A horizontal magnetic field on $M$ hence is written as a 2-form as follows
\begin{equation} \label{eq:betaintro}
\beta=\beta_{1}\nu_{1}\wedge \omega+\beta_{2}\nu_{2}\wedge \omega
\end{equation}
which is closed in the ordinary sense, or equivalently, with respect to $\diff_{H}$. The closure condition, of course, implies that the two coefficients in \eqref{eq:betaintro} are not independent.

Given a horizontal magnetic field  $\B$ on $M$, let $A$ be a horizontal magnetic potential on $M$ for $\beta$, i.e., such that $\beta=\diff_{H}A$. We stress that the choice of $A$ is not unique.   Sub-Riemannian magnetic geodesics are projections of integral curves of the Hamiltonian vector field associated with the magnetic Hamiltonian (here $q$ is a scalar denoting the charge of the particle)
\begin{equation}\label{eq:accam}
H_{A}=\frac{1}{2}\sum_{i=1}^{2}(h_{i}+q\,A_{i})^{2},
\end{equation}
where $A_{i}=\langle A,X_{i} \rangle$, cf.\ formula \eqref{eq:Ham0} for the classical case. The flow associated with the corresponding Hamiltonian vector field $ \vec{H}_{A}$ is called sub-Riemannian magnetic flow.

 Our first result is the characterization of sub-Riemannian magnetic geodesics, showing that they depend on the horizontal magnetic field $\beta$ and not on the particular choice of potential $A$. This should be compared with formula \eqref{eq:Omega}.

\begin{proposition} \label{p:maggeo}
Let $\B=\B_{1}\nu_1\wedge\omega + \B_{2}\nu_2\wedge\omega\in\Omega^2_H(M)$ be a horizontal magnetic field on $M$. Let $\gamma:[0,T]\to M$ be a sub-Riemannian magnetic geodesic on $M$ with charge $q$. Then there exists a smooth function $\alpha$ defined on $\gamma$ such that with respect to the Tanno connection $\nabla$
\begin{align}
\nabla_{\dot\gamma}\dot \gamma &= \alpha \,J\dot \gamma, \label{eq:dbmf1}\\
\nabla_{\dot \gamma} \alpha &= g(\tau(\dot \gamma),\dot \gamma)+q\, b(\dot \gamma), \label{eq:dbmf2}
\end{align}
where we denoted $b(\dot \gamma)=\beta(\dot \gamma,X_{0})$ and $\tau(\dot \gamma)=\Tor(X_{0},\dot \gamma)$.
\end{proposition}

Here $\mathrm{Tor}$ is the torsion of the Tanno connection $\nabla$. We refer to Section~\ref{s:tanno} for more details. Notice that for $\beta=0$ we recover the geodesic equation (cf.\ Proposition~\ref{p:77}).

 It is interesting to observe that in the \sr case the magnetic field does not appear in equation \eqref{eq:dbmf1} as in the Riemannian case, cf.\  \eqref{eq:Omega}, but rather in equation \eqref{eq:dbmf2}. This is coherent with the fact that $\beta$ is a Rumin differential, which on 1-forms acts as a differential operator of order two. See also Remark~\ref{r:dbrmf}. 
 
 Another interesting observation is that we can combine equations \eqref{eq:dbmf1} and \eqref{eq:dbmf2} as a single equation as follows
\begin{equation}\label{eq:db9}
\frac{d}{dt} g(\nabla_{\dot\gamma}\dot \gamma,J\dot \gamma)-g(\tau(\dot \gamma),\dot \gamma)=q\, b(\dot \gamma),
\end{equation}
and we can notice that the left-hand side of \eqref{eq:db9} coincides with the \emph{sub-Riemannian geodesic curvature} $k_{SR}(\gamma)$ introduced in \cite{BKACV} as a metric invariant of horizontal curves. Thus, also in the three-dimensional sub-Riemannian case we have an analog interpretation of the magnetic geodesic equation as a scalar equation to be compared with \eqref{eq:Omega2}:
\begin{equation}\label{eq:db10}
k_{SR}(\gamma)=q\,  b(\dot \gamma).
\end{equation}
The sub-Riemannian geodesic curvature can be defined as a metric invariant as follows: if $\gamma$ is a horizontal curve on $M$ and $d_{SR}$ denotes the sub-Riemannian  distance then \cite[Theorem~3]{BKACV} 
\begin{equation}\label{eq:db15}
d_{SR}^{2}(\gamma(t+\eps),\gamma(t))=\eps^{2}-\frac{k^{2}_{SR}(\gamma(t))}{6!}\eps^{6}+o(\eps^{6}).
\end{equation}
It is interesting to notice that this invariant in dimension three appears at order $6$, and not at order $4$ as in the Riemannian case, cf.\ \eqref{eq:db16}. 
Being the Rumin differential a derivation of second order in dimension three, one could interpret \eqref{eq:db10} as equality between second order invariants.

\subsection{The lifted \sr structure and its geometry}
We now move to the geometric interpretation of the sub-Riemannian magnetic geodesic. As in the classical case, a \sr magnetic geodesic can be interpreted as the projection of a suitable geodesic flow on a lifted \sr structure.
Let us start from the simplest example.
\begin{example}[from Heisenberg to Engel]
\label{ex:fromHtoEintro}
In the specific case of the Heisenberg group  $\mathbb H=\R^3$ with coordinates $(x,y,z)$, we have that the dual basis of 1-forms to $X_{1},X_{2},\partial_{z}$ is given by $dx,dy,\omega=dz-\frac{1}{2}(xdy-ydx)$. The Rumin complex in this case is given by
\begin{equation}
\Omega^1_H(\mathbb{H})=\mathrm{span}\{dx,dy\},
\qquad
\Omega^2_H(\mathbb{H})=\mathrm{span}\{dx\wedge\omega,dy\wedge\omega\}.
\end{equation}
Let us consider the following magnetic potential and the corresponding magnetic field
\begin{equation}
A=\frac{x^2}{2}dy, \qquad \B=\diff_{H}A=dx\wedge\omega .
\end{equation}
We stress that $A$ has a polynomial coefficient of degree 2 but $\beta$ is a 2-form with constant coefficients, since $\diff_{H}$ is a derivation of degree 2 on 1-forms.

One can then follow the Riemannian lift procedure described above by considering a \sr structure on $\R^4$ generated by the vector fields
\begin{equation}
Y_1=X_1+A(X_1)\p_w=\p_x-\dfrac{y}{2}\p _z,\qquad Y_2=X_2+A(X_2)\p_w=\p _y+\dfrac{x}{2}\p _z+\frac{x^2}{2}\p_w.
\end{equation} 
One can immediately notice that $[Y_1,Y_2]=\p_z+x\p_w$,  
and moreover
\begin{equation}
[Y_1,[Y_1,Y_2]]=\p_w, \qquad [Y_2,[Y_1,Y_2]]=0.
\end{equation}
The distribution $\lift\distr=\mathrm{span}\{Y_1,Y_2\}$ is bracket generating on $\R^4$. Requiring that $\{Y_1,Y_2\}$ form an orthonormal frame on $\lift\distr$, the lifted \sr structure is the Engel group. 
\end{example}

For a general contact manifold, given a magnetic potential $A\in \Omega^{1}_{H}(M)$ for a magnetic field $\B\in\Omega^2_H(M)$ and a basis $X_1,X_2$ for the contact distribution which is orthonormal for the metric on $\distr$, we define the lifted distribution 
\begin{equation}
	\lift \distr=\mathrm{span}\{Y_1,Y_{2}\}, \qquad Y_{i}=X_i+ A (X_i)\p_w.
\end{equation}
We notice that this is a rank 2 distribution in the 4-dimensional space $\lift M=M\times \R$, hence it has necessarily step $\geq 3$. One can show that the step is finite, i.e., the sub-Riemannian structure on $\lift M$ is well-defined, if at every point there exists a horizontal derivation of $\beta$ which is non-vanishing (cf.\ \cref{lem:step}). In the rest of the introduction we always make this assumption. 


  The following result completely characterizes the situation when the magnetic field is non-vanish\-ing. This is proved combining Proposition~\ref{p:normal}, Lemma~\ref{lem:step3}, Propositions~\ref{lem:abnormalsinBneq0} and \ref{p:dbp}.
 \begin{theorem} 
\label{intro:nonvanishing}
Let  $\B\in\Omega^2_H(M)$ be a horizontal magnetic field on a three-dimensional contact \sr manifold $M$ which defines the lifted sub-Riemannian structure on $\lift M$.  Then
\begin{itemize}
\item[(a)] normal extremal trajectories of $\lift M$ projects on magnetic geodesics on $M$,
\item[(b)] abnormal extremal trajectories of $\lift M$ projects on horizontal curves $\gamma$ on $M$ satisfying $\iota_{\dot \gamma}\beta=\beta(\dot \gamma,\cdot)=0$.
\end{itemize}
Moreover, in the region where $\B$ is non-vanishing, the lifted distribution $\lift \distr$  has step $3$ on $\lift M$, and for every $\lift p\in \lift M$ there exists a unique abnormal extremal trajectory $\lift \g$ passing through $\lift p$.
\end{theorem}
The last property can be rephrased by saying that, where $\B$ is non-vanishing, the lifted distribution $\lift \distr$ on $\lift M$  is a \emph{Engel-type} distribution.

 Observe that with respect to the magnetic lift of a Riemannian manifold, in this case abnormal extremal trajectories always exist, also when the magnetic field is never vanishing.  

\begin{remark} \label{r:intdb}
In the region where the magnetic field $\beta$ is non-vanishing, choosing a frame $X_1,X_2$ for $\distr$ and the corresponding dual elements, we can write $\B=\mg$ with $(\beta_{1},\beta_{2})\neq (0,0)$. The characteristic curves of $\beta$ (namely $\iota_{\dot\gamma} \beta =0$)
in this frame satisfy $\dot \gamma=-\beta_{2}X_{1}+\beta_{1}X_{2}$. 

Notice that, clearly, condition  $\iota_{\dot\gamma} \beta =0$ holds also for horizontal curves contained in the zero locus of $\beta$.
We stress that in both cases one has  $b(\dot \gamma)=0$, where $b$ appears in \eqref{eq:dbmf2}.
One can interpret this fact by saying that, along the projection onto $M$ of an abnormal extremal trajectory living on $\lift M$,  the magnetic field does not affect the dynamics.

\end{remark}

\subsection{On the zero locus}
Now we move to the study of the step of $\lift \distr$ in the lift of the zero locus $\mathcal{Z}$ of a magnetic field $\B\in\Omega^2_H(M)$. Let us write $\B=\mg$  with respect to a basis and let us consider the pair $(\beta_{1},\beta_{2})$ as a map from $M$ to $\R^{2}$. The step of the lifted distribution $\lift \distr$ at points in $\lift M$ is related with the rank of this map. 

\begin{theorem}
\label{intro:step-rank}
Let   $\B=\mg\in\Omega^2_H(M)$ be a horizontal magnetic field on $M$, a three-dimensional contact \sr manifold and let $p\in \mathcal{Z}$. 
\begin{enumerate}[(i)]
\item \label{itemintro:step2} If $\mathrm{rank}\left(\diff \B_1,\diff \B_2\right)|_{p}=2$, then the step of $\lift \distr$ at $\{ p \}\times \R \subset \lift M$ is $4$;
\item \label{itemintro:step1} If $\mathrm{rank}\left(\diff \B_1,\diff \B_2\right)|_{p}=1$, then the step of $\lift \distr$ at $\{ p \}\times \R \subset \lift M$ is $4$ or $5$;
\item \label{itemintro:step0} If $\mathrm{rank}\left(\diff \B_1,\diff \B_2\right)|_{p}=0$, then the step of $\lift \distr$ at $\{ p \}\times \R \subset \lift M$ is $\geq 5$.
\end{enumerate}
\end{theorem}
We notice that the result is independent of the choice of the frame.
When the rank of the map is 2, namely case (i), then the zero locus of the magnetic field is a curve. It is interesting to stress here that the step is not affected by the curve being horizontal or transversal to the distribution $\distr$ on $M$.

This is in contrast with case (ii) when the rank is 1. Up to switching indices, we can assume that $\diff \B_1\neq 0$.
 Hence $\Sigma=\beta_{1}^{-1}(0)$ defines a regular surface that contains the zero locus $\mathcal Z$ of $\beta$. We can characterize the step of the lifted \sr structure via a geometric analysis involving {\em characteristic points} of $\Sigma$, i.e., points $p\in \Sigma$ where  $\distr_p=T_p\Sigma$. 
\begin{theorem}\label{intro:rank1}
Let  $\B=\mg\in\Omega^2_H$ be a horizontal magnetic field on $M$, a three-dimensional contact \sr manifold satisfying $\mathrm{rank}\left(\diff\B_1,\diff\B_2\right)=1$ at $p\in \mathcal{Z}$.
Assume $\restr{\diff \B_1}{p}\neq 0$
and let $\Sigma=\B_1^{-1}(0)$.
Then $\lift\distr$ has step $5$ at $\{p\}\times \R$ if and only if $p$ is a characteristic point in $\Sigma$.
\end{theorem}

Finally, we stress that at points where the rank is equal to zero we can have arbitrarily large step. This can be shown by considering explicit magnetic fields of the following form for any given integer $n\geq 1$ 
\begin{equation}
\B=\frac{f^n}{n!}(b_1\nu_1+b_2\nu_2)\wedge\omega\in \Omega^2_H,
\end{equation}
where $b_1,b_2,f\in C^\infty(M)$ with $b_1^2+b_2^2\neq 0$ and $\diff f\neq 0$.
In this case the zero locus of $\B$ given by $\mathcal{Z}=f^{-1}(0)$ and one can prove that the step of the lifted distribution $\lift\distr$ is equal to $3$ at points which projects on $M \setminus \mathcal{Z}$, it is $n+3$ at points which projects on $\mathcal{Z}\setminus \mathrm{Char}(\mathcal{Z})$, it is equal to $2n+3$ at points which projects on $ \mathrm{Char}(\mathcal{Z})$, where $ \mathrm{Char}(\mathcal{Z})$ denotes the set of characteristic points in $\mathcal{Z}$. Cf.\ Lemmas~\ref{lem:example} and \ref{lem:examplestep}.

\subsection{Structure of the paper} 
Section~\ref{s:prel} is devoted to some preliminaries and in particular the description of the Rumin complex for three-dimensional manifolds. In Section~\ref{s:mfield} we introduce magnetic fields for three-dimensional sub-Riemannian manifolds and we obtain the magnetic geodesic equations. Then we define the associated lifted sub-Riemannian structure in Section~\ref{s:nonzeromf} and describe its properties in the case when the magnetic field is never vanishing. Section~\ref{s:zeromf} covers the case when the magnetic field can vanish somewhere and the characterization of the step of the lifted structure and its abnormal curves in this case. 
Finally in \cref{s:examples} we present a class of magnetic fields with a given surface as zero locus and we analyze the behavior of abnormal extremal trajectories with respect to the zero locus of the magnetic field.

%

\section{Preliminaries} \label{s:prel}

We first recall here some basic notions of \sr geometry, for a more detailed introduction we refer to \cite{thebook}.
Then we focus on the specific case of three-dimensional contact \sr manifolds.
Finally, we present the Rumin complex defined for three-dimensional contact \sr manifolds, cf.\ \cite{RuminContact}.

\subsection{Sub-{R}iemannian structures}
Let $M$ be a smooth, connected $n$-dimensional manifold equipped with $\distr$ a bracket generating vector distribution of rank $r$ in $TM$.

\begin{definition}
\label{def:bracketgenerating}
For $k\geq 1$,  we set $\distr^1=\distr$ and recursively define
\begin{equation}
\distr ^{k+1} = \mathrm{span} \left\{\distr^{k},\left[\distr,\distr^k \right]\right\}.
\end{equation}
We say that the distribution $\distr$ at a point $p\in M$ is bracket generating of step $k$ if it holds that
\begin{equation}
\distr ^{k}_p =T_p M 
\qquad \mbox{and} \qquad
\distr ^{j}_p \subsetneq T_p M \quad \forall\, j\in\{1,\ldots,k-1\}.
\end{equation}
Finally, the distribution is bracket generating if it is bracket generating at every $p\in M$.
\end{definition}

The manifold $M$, equipped with $\distr$ a bracket generating vector distribution of rank $r$ and with $g$ a metric defined on $\distr$, is said to be a \emph{smooth \sr manifold of rank $r$}. 

Let $X_1,\ldots,X_r$ a locally defined frame for $\distr$. 
A curve $\gamma : [0,T] \to M$ is \emph{horizontal} if it is absolutely continuous and there exists a \emph{control} $u\in L^\infty([0,T],\R^r)$ satisfying
\begin{equation}
\label{eq:control}
\dot\gamma(t)=\sum_{i=1}^r u_i(t) X_i(\gamma(t))\in \distr_{\gamma(t)}, \qquad \text{for a.e.}\ t \in [0,T].
\end{equation}
The \emph{length} of a horizontal curve is defined as:
\begin{equation}
\ell(\gamma) := \int_0^T \sqrt{g(\dot\gamma(t),\dot\gamma(t))} \de t.
\end{equation}
Finally, the \emph{\sr distance} $d_{SR}$ between any two points $p,q\in
 M$ is defined as
\begin{equation}\label{eq:infimo}
d_{SR} (p,q) := \inf\{\ell(\gamma) \mid  \gamma \text{ horizontal curve joining $p$ and $q$} \}.
\end{equation}
By the Chow-Rashevskii theorem, the bracket generating assumption ensures that the distance $d_{SR}$ is finite, continuous and it induces the manifold topology \cite[Chapter 3]{thebook}.

The \sr Hamiltonian is the smooth function $H:T^*M\to\R$ defined as
\begin{equation}
\label{eq:hamiltonian}
 H(\lambda)=\frac12\sum_{i=1}^r\la \lambda ,X_i(x)\ra ^2,
\end{equation} 
 where $\lambda\in T^*_xM$ and $\la\cdot,\cdot\ra$ is the usual duality pairing. 

Let $\vec{H}$ be the Hamiltonian vector field on $T^*M$ associated to $H$. Namely, $\vec{H}$ is the unique vector field such that 
\begin{equation}\label{eq:db15H}
\sigma(\cdot,\vec H) = \diff H,
\end{equation}
where $\sigma$ is the canonical symplectic form on the cotangent bundle.
A curve $\lambda:[0,T]\to T^*M$ that solves 
\begin{equation}
\label{eq:hamiltoniandynamics}
	\dot\lambda (t) = \vec{H}(\lambda(t)),
\end{equation} 
is called \emph{normal extremal}.
The projection of a normal extremal is called \emph{normal extremal trajectory}, it is a smooth curve parametrized with constant speed, and whose sufficiently small arcs are length-minimizers \cite[Chapter 4]{thebook}.

Not all length-minimizers arise projecting the Hamiltonian flow \eqref{eq:hamiltoniandynamics}. In \sr geometry one should also consider the so-called \emph{abnormal extremal trajectories}. 
These are horizontal curves that are critical points of the end-point map and depend only on the distribution $\distr$ (while do not depend on the choice of the metric $g$ on it). 

Length-minimizers of abnormal type admit a non zero lift $\lambda:[0,T]\to T^*M$ satisfying a time-dependent Hamiltonian equation and are contained in $\distr^{\perp}:=H^{-1}(0)$. For distributions of constant rank (all distributions appearing in this paper satisfy this assumption) we have the following characterization which can be found in \cite[Section~4.3]{thebook}.
\begin{proposition} \label{p:sympzero}
Let $M$ be a sub-Riemannian structure associated with a distribution $\distr$ of constant rank. Then a never vanishing Lipschitz curve in $H^{-1}(0)$ is a characteristic curve for $\sigma|_{H^{-1}(0)}$ if and only if it is the reparametrization of an abnormal extremal.
\end{proposition}

\subsection{Three-dimensional contact sub-{R}iemannian spaces}
\label{prel:contact}

A three-dimensional smooth manifold $M$ is a \emph{\sr contact} manifold if it is equipped with a smooth one-form $\omega$ satisfying the non degeneracy condition $\omega\wedge\diff\omega\neq 0$,
where $\diff$ denotes the exterior derivative.

The \sr structure on $M$ is defined by
(i) $\distr=\ker\omega$ a two-dimensional bracket generating distribution,
(ii) $g$ a smooth metric defined on $\distr$.  

We require that the two-dimensio\-nal volume form $\mathrm{vol}_g$ defined by $g$ on $\distr$ coincides with the two-dimensional form $-\diff \omega|_{\distr}$.
The \emph{Reeb vector field} $X_{0}$ is the unique vector field such that
\begin{equation}
\label{eq:reeb}
\iota_{X_{0}}\omega=1,\qquad \iota_{X_{0}}\diff\omega=0,
\end{equation}
where $\iota_{X}$ denotes the interior product with respect to the vector field $X$. 

Given any local orthonormal frame $\{X_1, X_2\}$ for $\distr$, the family $\{X_1, X_2,X_0\}$ is a local frame for $TM$. We set $\{\nu_1,\nu_2,\nu_{0}\}$ to be the corresponding dual frame, where $\nu_{0}=\omega$ is  the contact form.
By our convention, a horizontal frame is positively oriented if $-\diff\omega(X_1,X_2)=1$. For such frames, the commutation relations read:
\begin{equation}\label{eq:structurecoeff}
\begin{split}
\left[X_{1},X_{2}\right] & =c_{12}^{1}X_{1}+c_{12}^{2}X_{2}+X_{0}, \\
\left[X_{1},X_{0}\right] & =c_{10}^{1}X_{1}+c_{10}^{2}X_{2},\\
\left[X_{2},X_{0}\right] & =c_{20}^{1}X_{1}+c_{20}^{2}X_{2},
\end{split} 
\end{equation}
where the $c_{ij}^{k}$ are suitable smooth functions on $M$, called structure functions. Notice that the  structure functions are constant if the \sr structure is left-invariant on a Lie group (cf. \cite[Section~7.4]{thebook}).
\begin{example}
\label{ex:Heisenberg}
The prototype model of \sr contact manifold is the three-di\-men\-sio\-nal Heisenberg group $\mathbb{H}$. 
This consists in $M=\R^3$ with coordinates $(x,y,z)\in\R^3$, equipped with the globally defined contact form
\begin{equation}
\label{eq:contactformH}
\omega=dz-\frac{1}{2}(xdy-ydx).
\end{equation}
The distribution $\distr=\ker \omega$ is generated by the vector fields:
\begin{equation}
X_1=\p _x-\dfrac{y}{2}\p _z,  
\qquad 
X_2=\p _y+\dfrac{x}{2}\p _z.
\end{equation}
The Reeb vector field is $X_0=\p_z$. 
Setting $g$ the \sr metric on $\distr$ such that $\{X_1,X_2\}$ is an orthonormal frame for $\distr$ we have $\mathrm{vol}_g=-\restr{\diff\omega}{\distr}=dx\wedge dy$,  and  the following bracket relations hold true:
\begin{equation}
[X_1,X_2]=X_0, \qquad [X_1,X_0]=[X_2,X_0]=0.
\end{equation}
In conclusion, $\{X_1,X_2,X_0\}$ constitutes a frame for $TM$ with corresponding dual frame for $T^*M$ given by $\nu_1=dx, \nu_2=dy, \nu_{0}=\omega$.

\end{example}

It is convenient for later purposes to introduce $h_{i}:T^{*}M\to \R$  for $i=0,1,2$ a set of linear on fiber functions defined by $h_{i}(\lambda)=\langle \lambda,X_{i}\rangle$ associated with the frame $X_{i}$ for $i=0,1,2$. We have the following identities, where $\vec H$ denotes the Hamiltonian vector field in \eqref{eq:db15H}
$$H=\frac{1}{2}\sum_{i=1}^{2}h_{i}^{2},\qquad \vec H=\sum_{i=1}^{2}h_{i}\vec h_{i}=\sum_{i=1}^{2}h_{i}X_{i}+h_{i}(c_{ij}^{k}h_{k}+c_{ij}^{0}h_{0})\partial_{h_{j}}$$
In particular, treating the functions $(h_{1},h_{2},h_{0})$ as coordinates on the fibers, and given an integral curve of $\vec H$ in $T^{*}M$ satisfying \eqref{eq:hamiltoniandynamics}, we have that for its projection $\gamma$ on the base manifold $M$ 
$$\dot \gamma=\sum_{i=1}^{2}h_{i}X_{i},$$
while the functions $h_{i}$ satisfy for $j=0,1,2$
$$\dot h_{j}=\{H,h_{j}\}=\sum_{i,k=1,2} c_{ij}^{k}h_{i}h_{k}+c_{0j}^{k}h_{0}h_{k}$$
\begin{remark}\label{r:prope} It is useful to recall the following properties relating the Hamiltonian vector field $\vec h\in \mathrm{Vec}(T^{*}M)$ associated to a Hamiltonian $h\in C^{\infty}(T^{*}M)$ and the Poisson bracket. For a proof of these basic facts in symplectic geometry we refer to \cite[Chapter 4]{thebook}.
\begin{itemize}
\item[(a)] If $h,k,k'\in C^{\infty}(T^{*}M)$ then $\vec h (k)=\{h,k\}$ and $\{h,kk'\}=\{h,k\}k'+\{h,k'\}k$. 
\item[(b)] If $h_{X}=\langle \lambda, X\rangle$ and $h_{Y}=\langle \lambda, Y\rangle$ are linear on fibers functions then $\{h_{X},h_{Y}\}=h_{[X,Y]}$.
\item[(c)] If $h_{X}=\langle \lambda, X\rangle$ and $\alpha, \alpha' \in C^{\infty}(M)$ then $\{h_{X},\alpha\}=X\alpha$ and $\{\alpha,\alpha'\}=0$.\end{itemize}
Thanks to properties (a)-(c) one can prove in particular that for every function $\alpha\in C^{\infty}(M)$ (regarded as a function in $T^{*}M$) the vector field $\vec \alpha$ is vertical, namely that $\pi_{*}\vec \alpha=0$.
\end{remark}
We recall the \emph{Cartan formula} for differential one-forms:
for any smooth one-form $\tau$ on $M$ and any smooth vector fields $V,W$, it holds that:
\begin{equation}
\label{eq:cartan}
\diff \tau(V,W)+\tau([V,W])=V(\tau(W))-W(\tau(V)).  
\end{equation} 
We deduce that for $k=0,1,2$ we have
\begin{equation}
\label{eq:cartanckij}
\diff\nu^k=-c^k_{10}\nu_1\wedge\omega -c^k_{20}\nu_2\wedge\omega -c^k_{12}\nu_1\wedge\nu_2.
\end{equation}

The volume form $\nu_1\wedge\nu_2\wedge\omega$ coincides with the so-called \emph{Popp's volume}. 
Moreover, since $\diff\omega=-\nu_1\wedge\nu_2$ by \eqref{eq:cartanckij}, notice that the Popp's volume coincides with $-\omega\wedge\diff\omega$. We refer to \cite{Popp} for more details on the construction in more general sub-Riemannian manifolds.
Denoting with $\mu$ the smooth measure associated with the Popp's volume, 
the divergence of a vector field $X$ with respect to $\mu$ is the smooth function $\mathrm{div}_\mu(X)$ that satisfies
\begin{equation}
	\mathcal{L}_X\mu=\mathrm{div}_\mu(X)\mu,
\end{equation}
where $\mathcal{L}_X$ denotes the Lie derivative with respect to $X$.
Recalling that $\mathcal{L}_X=\diff\circ \iota_{X}+\iota_{X}\circ \diff$ (where $\iota_{X}$ denotes the interior product with $X$) and that volume forms are closed, we have 
\begin{equation}
\label{Cartanvolume}
\mathcal{L}_X\mu=-\diff\left(\iota_X\,\omega\wedge\diff\omega\right).
\end{equation}
Moreover, using also \eqref{eq:cartanckij}, one computes the divergence in terms of the coefficients in \eqref{eq:structurecoeff}:
\begin{equation}
\label{eq:divXi}
\mathrm{div}_\mu(X_1)=-c^2_{12}, \qquad \mathrm{div}_\mu(X_2)=c^1_{12}.
\end{equation}
The divergence operator is linear and satisfies the following Leibniz rule:
\begin{equation}
\label{leibniz}
\mathrm{div}_\mu(fX)=Xf+f\mathrm{div}_\mu(X),
\end{equation}
where $f:M\to\R$ is smooth and $X$ is a vector field on $M$.

\begin{remark} We end this section by recalling that the canonical symplectic form $\sigma$ on the cotangent bundle $T^{*}M$ can be written in the basis of 1-forms $\nu_{i}$ dual to the chosen basis of $TM$ and the differentials $dh_{i}$ for $i=0,1,2$ of the linear on fibers functions defined above.
The symplectic  form $\sigma$ is the differential of the tautological form $s=\sum_{i=1,2,0} h_{i}\nu_{i}$ hence
$$\sigma=ds=\sum_{i=1,2,0} dh_{i}\wedge \nu_{i} + h_{i}d\nu_{i}$$
We recall that contact structures on three-dimensional manifolds do not admit non-constant abnormal extremals trajectories (see \cite[Proposition~4.38]{thebook}).
\end{remark}

\subsection{The {R}umin complex}
\label{prel:Rumin}
The construction of the {R}umin complex we present here is specific for three-dimensional \sr contact manifolds. For simplicity, we define objects globally on $M$ but everything can be defined on open sets $U\subset M$.

\begin{definition}[Horizontal differential forms]
For $k=0,\dots,3$, 
we define the space of the \emph{horizontal $k$-forms} $\Omega^k_H(M)$ as
\begin{equation}
\begin{split}
\Omega^0_H(M)&:=C^\infty(M),\\
\Omega^1_H(M) &:=\mathrm{span}\{\nu_1,\nu_2\}, 
\\
\Omega^2_H(M)&:=\mathrm{span}\{\nu_1\wedge\omega,\nu_2\wedge\omega\},\\
\Omega^3_H(M)&:=\mathrm{span}\{\omega\wedge\diff \omega\},
\end{split}
\end{equation}
where $\nu_1,\nu_2,\omega$ is the dual basis to  $X_1,X_2,X_0$.
\end{definition}

\begin{definition}[Horizontal differential]
For $k=0,1,2$, we define the \emph{horizontal differential} or \emph{Rumin differential}
\begin{equation}
\diff^k_H:\Omega^k_H(M)\to \Omega^{k+1}_H(M)
\end{equation}
in the following way:
\begin{itemize}
\item[(i)] For $f\in \Omega_H^0(M)$, we set
	\begin{equation}	\label{eq:dH0}
		\diff^0_H f 
		=(X_1f)\nu_1+(X_2f)\nu_2;
	\end{equation}
\item[(ii)] For $A=A_1\nu_1+A_2\nu_2\in \Omega_H^1(M)$, we set
	\begin{equation}	\label{eq:dH1}
		\diff^1_HA=\diff\left(A+\diff A(X_1,X_2)\omega\right);
	\end{equation}
\item[(iii)] For $\B=\mg\in \Omega^2_H(M)$, we set
	\begin{equation}	\label{eq:dH2}
		\diff^2_H \B=\diff\left(\mg\right).
	\end{equation}
\end{itemize}
\end{definition}
 Since  for $k=1,2$ it holds $\diff_H^k\circ\diff_H^{k-1}=0$, the co-chain 
 \begin{equation}
0 \longrightarrow \R\longrightarrow 
 \Omega^0_H(M)
\stackrel{\diff_H^0}{\longrightarrow} 
 \Omega^1_H(M)
\stackrel{\diff^1_H}{\longrightarrow} 
\Omega^2_H(M)
\stackrel{\diff_H^2}{\longrightarrow}  
\Omega^3_H(M)
\longrightarrow 0.
\end{equation}
defines a complex  $\left(\Omega_H^*,\diff^*_H\right)$, referred to as the \emph{Rumin complex}. 

Furthermore, $\diff_H^0$ and $\diff_H^2$ are differential operators of the first order, whereas $\diff_H^1$ is a second-order differential operator, as we show in the following with more explicit computations. In what follows, the index $k$ in the notation $\diff_H^k$ can be omitted whenever it is clear from the context.

\begin{lemma}
\label{lem:d_H}
Let $A=A_1\nu_1+A_2\nu_2\in \Omega^1_H(M)$ such that $\beta=\diff_H A=\mg \in \Omega^2_H(M)$. Then,
\begin{align}
\begin{split}\label{eq:B_1}
\beta_1&=X_1X_2(A_1)-X_1X_1(A_2)-c^1_{12}X_1(A_1)-c^2_{12}X_1(A_2)-X_0(A_1)
\\ & \quad-(X_1(c^1_{12})+c^1_{10})A_1-(X_1(c^2_{12})+c^2_{10})A_2, \end{split}
\\ \begin{split} \label{eq:B_2}
\beta_2&= X_2X_2(A_1)-X_2X_1(A_2)-c^1_{12}X_2(A_1)-c^2_{12}X_2(A_2)-X_0(A_2)
\\ & \quad-(X_2(c^1_{12})+c^1_{02})A_1-(X_2(c^2_{12})+c^2_{02})A_2. \end{split}
\end{align}
\end{lemma}

\begin{proof} 
Since $\B_i=\diff_HA(X_i,X_0)$ for $i=1,2$, by \eqref{eq:dH1} and using \eqref{eq:cartanckij} it holds that
\begin{align}
\beta_1&=X_1(\diff A(X_1,X_2))-X_0(A_1)-c^1_{10}A_1-c^2_{10}A_2, \label{eq:B1}
\\ \beta_2&=X_2(\diff A(X_1,X_2))-X_0(A_2)-c^1_{20}A_1-c^2_{20}A_2.\label{eq:B2}
\end{align} 
Exploiting \eqref{eq:cartan}, we obtain that 
\begin{equation}
\begin{split}
\diff A(X_1,X_2)&=X_1(A(X_2))-X_2(A(X_1))-A([X_1,X_2])
\\&=X_1(A_2)-X_2(A_1)-c^1_{12}A_1-c^2_{12}A_2.
\end{split}
\end{equation}
Substituting in \eqref{eq:B1} and \eqref{eq:B2} we conclude.
\end{proof}

The following result, due to Rumin, guarantees that every $\diff_H$-closed horizontal form is locally $\diff_H$-exact and that the cohomology induced by $(\Omega_H, \diff_H)$ is isomorphic to the de Rham cohomology of $M$. Consequently, all constructions involving the Rumin complex can also be interpreted locally, namely on an open neighborhood $U$ that is topologically equivalent to $\mathbb{R}^3$.

\begin{proposition}[\cite{RuminContact}]
\label{prop:Rumin}
The Rumin complex $\left(\Omega_H,\diff_H\right)$ is locally exact
and it computes the de Rham cohomology of $M$.
\end{proposition}

Originally, the Rumin complex was introduced by Rumin in the framework of contact manifolds \cites{RuminContact}, and later for Carnot groups \cites{Rum00a}. 
More recently, an alternative construction of this complex has been proposed in \cites{FT23,FT24} including homogeneous groups and filtered manifolds, of which equiregular \sr structures represent particular cases.

\subsection{Tanno connection} \label{s:tanno} 
We introduce the Tanno connection, which is a canonical connection on sub-Riemannian contact manifold $M$, see  \cites{tanno, abrcontact}. For proofs of some facts listed in this subsection, in the same three-dimensional setting, we also refer to \cite{BBCCM}.

Given the normalized contact structure $\omega$ on $M$, we define the linear map $J:TM\to TM$ by $g(X,JY)=\mathrm{d}\omega(X,Y)$ for horizontal vector fields $X,Y$, while $JX_{0}=0$. 
\begin{definition}
\label{def:Tanno}
The Tanno connection $\nabla$ is the unique linear  connection on $TM$ satisfying
\begin{enumerate}[(i)]
\item $\nabla g=0$, $\nabla X_0=0$;
\item $\Tor(X,Y)=g(X,JY) X_0=\mathrm{d}\omega(X,Y)X_0$ for all $X,Y\in\distr$;
\item $\Tor(X_0,JX)=-J\Tor(X_0,X)$ for any vector field $X$ on $M$. 
\end{enumerate}
where $\Tor$ denotes the torsion of the connection $\nabla$.
\end{definition}
In what follows $\nabla$ will always denote the Tanno connection. Being $\nabla$ a metric connection, it follows that $\nabla_XY$ is parallel to $JY$ for any unitary horizontal vector fields $X,Y$. 
It is also easy to check that the Tanno connection $\nabla$ commutes with the operator $J$, i.e., $\nabla_XJY=J\nabla_XY.$

Let us introduce the ``horizontal'' Christoffel symbols
\begin{equation}\label{eq:db4bis}
\Gamma_{ij}^k := \frac{1}{2}\left( c_{ij}^k + c_{ki}^j + c_{kj}^i \right), \qquad i,j,k=1,2.
\end{equation}
Notice that $\Gamma_{ij}^k + \Gamma_{ik}^j=0$ and that one can recover some of the structural functions with the relation
\begin{equation}\label{eq:db4}
c_{ij}^k = \Gamma_{ij}^k - \Gamma_{ji}^k,  \qquad i,j,k=1,2.
\end{equation}
In terms of the structural functions, we have, for $i,j=1,2$
\begin{equation}\label{eq:db5}
\nabla_{X_i} X_j = \sum_{k=1}^{2} \Gamma_{ij}^k X_k, \qquad \nabla_{X_0} X_i = \frac{1}{2}\sum_{k=1}^{2}(c_{k0}^i - c_{i0}^k) X_k, \qquad JX_i = \sum_{j=1}^{2} c_{ij}^0 X_j.
\end{equation}

Let $\Tor$ be the torsion associated to the Tanno connection and  $\tau:\distr \to \distr$ be the linear operator $\tau(X):=\Tor(X_0,X)$ for $X$ horizontal. Then $\tau$ is symmetric with respect to $g$ and
\begin{equation}
\tau(X_i) = \frac{1}{2}\sum_{k=1}^{2}(c_{k0}^i + c_{i0}^k) X_k, \qquad \Tor(X_j,X_k) = - c_{jk}^0 X_0, \end{equation}
The equations for sub-Riemannian normal geodesics on $M$ can be also characterized as follows. 
\begin{proposition}\label{p:77}
Let $\gamma:[0,T]\to M$ be a sub-Riemannian geodesic on $M$. Then there exists a smooth function $\alpha$ defined on $\gamma$ such that
\begin{align}
\nabla_{\dot\gamma}\dot \gamma &= \alpha \,J\dot \gamma\\
\nabla_{\dot \gamma} \alpha &= g(\tau(\dot \gamma),\dot \gamma)
\end{align}
\end{proposition}
This result is well-known (see \cite[Proposition 15]{RuminContact}). A proof for general contact structures, with language similar to the one presented here, can be found for instance in \cite{abrcontact}. We refer to \cite{BKACV} for the relation with the geodesic curvature in the three-dimensional case. 
\begin{remark}
Notice that in the case of the Heisenberg group $\tau=0$, hence $\alpha$ is constant along geodesics $\gamma$ and actually coincides with the vertical part of the covector defining it.

In the general case the quantity appearing in the second equation $\eta(\dot \gamma):= g(\tau(\dot \gamma),\dot \gamma)$ is a directional invariant which is related to the metric invariant $\chi$ of the contact structure \cite{AB12}. We refer to \cite{BKACV} for more details.
\end{remark}

\newcommand{\bi}{\zeta}
\section{Magnetic fields on 3{D} contact sub-{R}iemannian manifolds}
\label{s:mfield}

Let us consider a three-dimensional contact \sr manifold, which we will denote by $M$ throughout the paper. We start by introducing horizontal magnetic fields.

\begin{definition}
A \emph{horizontal magnetic field} is a differential two-form $\B\in\Omega^2_H(M)$ which is closed with respect to the exterior derivative $\diff$. 
With respect to a choice of the frame, there exist $\beta_1,\beta_2\in C^\infty(M)$ such that
\begin{equation}
	\B=\B_{1}\nu_1\wedge\omega + \B_{2}\nu_2\wedge\omega, \qquad \diff \B=0.
\end{equation}
If $\B$ is exact, a  smooth one-form $A=A_1\nu_1+A_2\nu_2\in \Omega_H^1$ such that $\B=\diff_HA$ is called \emph{horizontal magnetic potential}.
\end{definition}

\begin{remark} By formula \eqref{eq:dH2}, one has $\diff\B=\diff_{H}\B$ for every $\B\in\Omega^2_H$, so that one can use equivalently both differentials in the above definition. In particular, since $\diff_{H}\B=0$, we can always find a horizontal magnetic potential $A\in\Omega^1_H(M)$ locally defined on $M$ (cf.\  \cref{prop:Rumin}).  
 Moreover, notice that $A\in\Omega^1_H(M)$ is unique up to adding a term $\diff_Hf$ with $f\in C^\infty(M)$.
\end{remark}
\subsection{Magnetic geodesics equation}
In this section we want to prove a characterization of magnetic geodesics in sub-Riemannian geometry similar to the one given in Proposition~\ref{p:77} for classical geodesics.

Let $\B\in\Omega^2_H(M)$ be a horizontal magnetic field on $M$ and let $A\in\Omega^1_H(M)$ be a horizontal magnetic potential on $M$ for $\beta$ such that $\beta=\diff_{H}A$. Sub-Riemannian magnetic geodesics are obtained as projections of integral curves of the Hamiltonian vector field associated with the magnetic Hamiltonian (here $q$ is a scalar denoting the charge of the particle)
\begin{equation}\label{eq:accam}
H_{A}=\frac{1}{2}\sum_{i=1}^{2}(h_{i}+q\,A_{i})^{2},
\end{equation}
that is obtained by modifying the sub-Riemannian Hamiltonian with the functions $A_{i}=\langle A,X_{i} \rangle$,  (cf.\ formula \eqref{eq:Ham0} for the classical case). 
The corresponding Hamiltonian vector field is written as
$$\vec{H}_{A}=\sum_{i=1}^{2}(h_{i}+q\, A_{i})(\vec h_{i}+q\, \vec A_{i}).$$
We give a geometric characterization of magnetic geodesics, showing that the notion we introduced of magnetic field is well-defined since geodesics do not depend on the choice of $A$.

\subsection{Proof of Proposition~\ref{p:maggeo}}
 By linearity with respect to $q$, it is enough to consider the case $q=1$. Consider  an integral curve $\dot \lambda=\vec H_{A}(\lambda)$ and set $\lambda=(\gamma,h)$ where $\gamma$ is the projection onto $M$ and  $h=(h_{1},h_{2},h_{0})$ are the vertical coordinates associated with the frame $X_{1},X_{2},X_{0}$ as discussed in Section~\ref{prel:contact}.

Thanks to Remark~\ref{r:prope}, being $A_{i}$  smooth functions on $M$, the vector fields $\vec A_{i}$ are vertical, i.e., $\pi_{*}\vec A_{i} =0$. Projecting the equation  $\dot \lambda=\vec H_{A}(\lambda)$ on $M$ one thus obtains
$$\dot \gamma=\sum_{i=1}^{2}(h_{i}+A_{i})X_{i},$$
where the quantities $h_{i}$ and $A_{i}$ satisfy the following differential equations along the flow (recall that $\dot f=\vec H_{A}(f)=\{H_{A},f\}$)
\begin{align} \label{eq:db1}
\dot h_{j}&=\sum_{i=1}^{2}(h_{i}+A_{i})\{h_{i}+A_{i},h_{j}\}=\sum_{i=1}^{2}(h_{i}+A_{i})[\{h_{i},h_{j}\}+\{A_{i},h_{j}\}]\\
&=\sum_{i,k=1}^{2}(h_{i}+A_{i})[c_{ij}^{k} h_{k}+c_{ij}^{0} h_{0}-X_{j}(A_{i})]. \nonumber
\end{align}
Similarly
\begin{align}\label{eq:db2}
\dot A_{j}&=\sum_{i=1}^{2}(h_{i}+A_{i})\{h_{i}+A_{i},A_{j}\}=\sum_{i=1}^{2}(h_{i}+A_{i})[\{h_{i},A_{j}\}+\{A_{i},A_{j}\}]\\
&=\sum_{i=1}^{2}(h_{i}+A_{i})[X_{i}(A_{j})]. \nonumber
\end{align}
Now we are ready to compute the quantity $\nabla_{\dot \gamma}\dot \gamma$
to characterize the magnetic geodesic equation. We have (recall that $\nabla_{\dot \gamma} f=\dot f$)
\begin{align}\label{eq:db3}
\nabla_{\dot \gamma}\dot \gamma&=\sum_{j=1}^{2}\nabla_{\dot \gamma}[(h_{j}+A_{j})X_{j}]=\sum_{j=1}^{2}(\dot h_{j}+\dot A_{j})X_{j} +(h_{j}+A_{j})\nabla_{\dot \gamma}X_{j}\\
&=\sum_{i,j=1}^{2}(\dot h_{j}+\dot A_{j})X_{j} +(h_{j}+A_{j})(h_{i}+A_{i})\sum_{k=1}^{2}\Gamma_{ij}^{k}X_{k}, \nonumber
\end{align}
where we used the relations $\nabla_{X_{i}}X_{j}=\sum_{k=1}^2\Gamma_{ij}^{k}X_{k}$ given in Section~\ref{s:tanno}. Adding \eqref{eq:db1} and \eqref{eq:db2}
we get
\begin{align*}
\dot h_{j}+\dot A_{j}&=\sum_{i,k=1}^{2}(h_{i}+A_{i})[c_{ij}^{k} h_{k}+ c_{ij}^{0} h_{0}+X_{i}(A_{j})-X_{j}(A_{i})],
\end{align*}
and replacing in \eqref{eq:db3} (also switching $j$ and $k$ in the second summand)
\begin{align*}
\nabla_{\dot \gamma}\dot \gamma
&=\sum_{i,j,k=1}^{2}(h_{i}+A_{i})[c_{ij}^{k} h_{k}+ c_{ij}^{0} h_{0}+X_{i}(A_{j})-X_{j}(A_{i})]X_{j} +(h_{k}+A_{k})(h_{i}+A_{i})\Gamma_{ik}^{j}X_{j}\\
&=\sum_{i,j,k=1}^{2}(h_{i}+A_{i})[-c_{ij}^{k} A_{k}+ c_{ij}^{0} h_{0}+X_{i}(A_{j})-X_{j}(A_{i})]X_{j} ,
\end{align*}
 where we added and subtracted $c_{ij}^{k} A_{k}$ in the first sum and used that for horizontal coefficients
$c_{ij}^{k}+\Gamma_{ik}^{j}$ 
is skew symmetric in $i,k$ due to \eqref{eq:db4bis} and \eqref{eq:db4}. Denoting
$$\bi_{ij}=dA(X_{i},X_{j})=X_{i}(A_{j})-X_{j}(A_{i})-\sum_{k=1}^{2}c_{ij}^{k}A_{k},$$
we have
\begin{align}\label{eq:db11}
\nabla_{\dot \gamma}\dot \gamma&=\sum_{i,j=1}^{2}(h_{i}+A_{i})[\bi_{ij} + c_{ij}^{0} h_{0}]X_{j} ,
\end{align}
which can be rewritten in a more compact form using \eqref{eq:db5} (notice that $\bi_{ij}=\bi c_{ij}^{0}$ where $\bi=\bi_{12}$)
\begin{align} \label{eq:db12}
\nabla_{\dot \gamma}\dot \gamma&=(\bi+h_{0})J\dot \gamma.
\end{align} 
Let us now compute first
\begin{align*}
\dot \bi=\nabla_{\dot \gamma} \bi&=\sum_{i=1}^{2}(h_{i}+A_{i})\{h_{i}+A_{i},\bi\}=\sum_{i=1}^{2}(h_{i}+A_{i})X_{i}(\bi)\\
&=\sum_{i=1}^{2}(h_{i}+A_{i})(\beta_{i}+X_{0}(A_{i})+c_{i0}^{j}A_{j}),
\end{align*}
where in the second line we crucially used relations \eqref{eq:B1} and \eqref{eq:B2}. Moreover 
\begin{align*}
\dot h_{0}=\nabla_{\dot \gamma} h_{0}&
=\sum_{i=1}^{2}(h_{i}+A_{i})[\{h_{i},h_{0}\}+{\{A_{i},h_{0}\}}]\\
&=\sum_{i,j=1}^{2}(h_{i}+A_{i})(c_{i0}^{j}h_{j}-X_{0}(A_{i})),
\end{align*}
and summing together we get
\begin{align*}
\dot \bi + \dot h_{0}
&=\sum_{i,j=1}^{2}\beta_{i} (h_{i}+A_{i})+c_{i0}^{j}(h_{j}+A_{j})(h_{i}+A_{i}).
\end{align*}
The conclusion follows recognising that given $v=\sum_{i=1}^{2}v_{i}X_{i}$, one has
\begin{equation}
b(v)=\beta(v,X_{0})=\sum_{i=1}^{2}\beta_{i}v_{i},\qquad 
g(\tau(v),v)=\frac{1}{2}\sum_{i,k=1}^{2}(c_{k0}^i + c_{i0}^k) v_{i}v_{k}. \qedhere
\end{equation}

\begin{remark} \label{r:dbrmf}
It is interesting to observe that in the Riemannian case since $\beta=dA$ involves the exterior differential, the magnetic field appears in the equation of the geodesics  at the level of equation \eqref{eq:db12} (or, equivalently, \eqref{eq:dbmf1}). 

In the sub-Riemannian case, we have $\beta=\diff_{H}A$ and the Rumin differential acts as a differential operator on degree $2$ on one-forms. Hence the coefficients of the magnetic field contain second order derivatives of the coefficients of $A$ and coherently appear in the second equation \eqref{eq:dbmf2}.
\end{remark}
\subsection{The lifted sub-Riemannian structure}
Inspired by the construction described in \cref{s:intro},
we associate with a horizontal magnetic field $\B$ a lifted \sr structure on the line-bundle $ \lift M = M\times \R$. 
The construction we describe assumes the existence of a horizontal magnetic potential $A\in\Omega^1_H(M)$ globally defined. We stress again that, if such a potential does not exist, the construction is still-well defined locally on $M$.

\smallskip

Let $\B$ be a horizontal magnetic field, and $A$ a horizontal magnetic potential for $\B$.
Given $\g:[0,T]\to M$ a horizontal curve in $M$, and $w_{0}\in \R$, we define the lift $\lift{\g}=(\g,w):[0,T]\to \lift{M}$ satisfying $w(0)=w_{0}$ with
\begin{equation}
w(t)=w_{0}+\int_0^tA(\dot{\g}(s))ds.
\end{equation}
Notice that in particular $\dot{\lift{\g}}$ is in the kernel of the 1-form $dw-A$. In particular if
\begin{equation}
\dot{ \g}(t)=u_1(t)X_1|_{\g (t)}+u_2(t)X_2 |_{\g (t)},
\end{equation}
then
\begin{equation}
\dot{\lift \g}(t)=u_1(t)\restr{(X_1+A(X_1)\p_w)}{\lift \g (t)}+u_2(t)\restr{(X_2+A(X_2)\p_w)}{\lift \g (t)},
\end{equation}
where we write elements of $T\lift M\simeq TM\oplus T\R$ as $X+\alpha \p_w$ with $X\in TM$ and $\alpha \in \R$. We want to set the length of $\lift \g$ to be equal to the length of $\g$. This is done in the following.

\begin{definition}
The \emph{lifted distribution} on $\lift M$ is the horizontal $\mathrm{rank}$-two distribution given by $\lift \distr = \mathrm{span}\{Y_1,Y_2\}$, where $Y_1,Y_2$ is an orthonormal frame defined by
\begin{equation}
\label{eq:generatingfamily}
Y_1=X_1+A(X_1)\p_w, \qquad Y_2=X_2+A(X_2)\p_w.
\end{equation} 
and $X_{1},X_{2}$ is an orthonormal frame for $\distr$.
\end{definition}

It is easy to see that the distribution $\lift \distr$ is independent of the choice of the frame for $\distr$. Indeed, by linearity of $A$, it holds that $$\lift \distr = \mathrm{span}\{X+A(X)\p_w\mid X\in \distr\}.$$ Moreover if $g$ denotes the inner product on $\distr$ and $\pi:\lift M \to M$ denotes the canonical projection the metric just defined on $\lift \distr$ is given by  $\lift g:=\pi^{*}g$.

\begin{remark} 
\label{rem:independenceofY_0}
Since $\distr$ is a contact distribution on $M$, the vector field $[Y_1,Y_2]$ is linearly independent from $\lift \distr$. This follows from the fact that
$$[Y_1,Y_2]=[X_1,X_2]+\left(X_1A_2-X_2A_1\right)\p_w,$$
and that $[X_1,X_2]$ is linearly independent from $X_{1},X_{2}$. 
More precisely, denoting
\begin{equation}
\label{eq:Y0}
Y_0= X_0+\diff A(X_1,X_2)\p_w,
\end{equation}
using \eqref{eq:structurecoeff} and \eqref{eq:cartan}, we obtain
\begin{equation}
\label{eq:[Y1,Y2]}
[Y_1,Y_2]=c_{12}^{1}Y_{1}+c_{12}^{2}Y_{2}+Y_0.
\end{equation}
We observe that $Y_0$ is not in the distribution $\lift \distr$ and it is independent of the choice of the orthonormal frame for $\distr$.
\end{remark}

\begin{remark}
\cref{ex:fromHtoEintro} shows how the above construction applied to $M=\mathbb H$ with the horizontal magnetic field $\B=\diff_{H}A=dx\wedge\omega$ associated with the horizontal magnetic potential $A=\frac{x^2}{2}dy$ gives rise to the Engel structure on $\R^4$.

It is interesting to compare the choice of $A$ of \cref{ex:fromHtoEintro} with the horizontal 1-form
\begin{equation}
	A'=-\left(\frac{z}{2}+\frac{xy}{12}\right)dx+\frac{x^2}{12}dy.
 \end{equation}
One can check that 
$$A-A'=\left(\frac{z}{2}+\frac{xy}{12}\right)dx+\frac{5x^2}{12}dy =\diff_H\left(\frac{xz}{2}+\frac{x^2y}{6}\right).$$ 
The lifted distribution $\lift \distr'$ associated to $A'$ differs from the lifted distribution $\lift\distr$ associated to $A$, nevertheless, the commutator relations recover the Lie algebra structure of the Engel group.

In particular, we have that $\lift \distr'$ is generated by
\begin{equation}
Y'_1
		=\p_x-\dfrac{y}{2}\p _z-\left(\frac{z}{2}+\frac{xy}{12}\right)\p_w, 
\qquad
Y'_2
		=\p _y+\dfrac{x}{2}\p _z+\frac{x^2}{12}\p_w.
\end{equation}
And it holds that $[Y'_1,Y'_2]=\p_z+\frac{x}{2}\p_w$, while $[Y'_1,[Y'_1,Y'_2]]=\p_w$ and $[Y'_2,[Y'_1,Y'_2]]=0$.
\end{remark} 
 One can prove, more in general, that given $A\in \Omega^{1}_{H}(M)$ and $A'=A+\diff_{H} f$, with $f\in C^\infty(M)$. 
Then, the Lie algebras $\mathrm{Lie}\,\lift \distr (A) $ and $\mathrm{Lie}\,\lift\distr(A')$ defined by the distribution associated with the two different potentials are isomorphic. 

\subsection{Normal extremals} In this section we prove the following fact, relating the \sr Hamiltonian flow on $\lift M$ and the magnetic flow on $M$.
\begin{proposition} \label{p:normal}
The \sr normal extremal trajectories for the lifted \sr structure $(\lift M, \lift \distr, \lift g)$ project onto magnetic geodesics on $M$.
\end{proposition}
\begin{proof} This is a direct consequence of the construction. We report here a sketch of the proof for completeness.  
In what follows we identify $\lift M=M\times \R$ where points are denoted by pairs $(p,w)$, and $T \lift M\simeq  TM \oplus T\R$. Similarly we can consider the identification $T^{*}\lift M\simeq  T^{*}M \oplus T^{*}\R$, denoting points $\zeta=(\lambda,\zeta_{w})\in T^{*}M \oplus T^{*}\R$.

The sub-Riemannian Hamiltonian associated with the lifted structure $(\lift M, \lift D, \lift g)$ is
\begin{equation}\label{eq:78}
\lift H=\frac{1}{2}\sum_{i=1,2}\langle\zeta,Y_{i}\rangle^{2}.
\end{equation}
Using the fact that $\zeta=(\lambda,\zeta_{w})$ and $Y_{i}=X_{i}+A_{i}\partial_{w}$ we can decompose
$$\langle\zeta,Y_{i}\rangle=\langle\lambda,X_{i}\rangle+A_{i}\langle \zeta, \partial_{w}\rangle=h_{i}+A_{i}\zeta_{w},$$
where $h_{i}=\langle\lambda,X_{i}\rangle$ is the Hamitonian linear on fibers on $T^{*}M$ (cf.\ Section~\ref{prel:contact}).
Hence we can rewrite \eqref{eq:78} as
\begin{equation}\label{eq:79}
\lift H=\frac{1}{2}\sum_{i=1,2} (h_{i}+A_{i}\zeta_{w})^{2}.
\end{equation}
Since the functions $A_{i}=A_{i}(p)$ do not depend on the variable $w$, the Hamiltonian $\lift H$ is independent of $w$ as well. This implies that along the corresponding Hamiltonian flow one has $\dot \zeta_{w}=0$. Setting $\zeta_{w}=q$ (this is the charge of the particle!) one can rewrite the Hamiltonian \eqref{eq:79} as
\begin{equation}
\lift H=\frac{1}{2}\sum_{i=1,2} (h_{i}+qA_{i})^{2}.
\end{equation}
Writing Hamiltonian equations and comparing with the magnetic Hamiltonian $H_{A}$ given in \eqref{eq:accam} it readily follows that projections on $M$ of integral curves of the Hamiltonian vector field associated with $\lift H$ on $\lift M$ are magnetic geodesics given in Proposition~\ref{p:maggeo}.
\end{proof}
\subsection{On the Maxwell equations}
The space of the Rumin two-forms $\Omega^2_H(M)$ has dimension two. The closure condition on the magnetic field provides a constraint, hence the coefficients of a magnetic field are not independent.  
In this section, we describe how these coefficients interact, starting from reinterpreting the closure condition in terms of a Maxwell-type equation.

\begin{lemma}
Let $\B=\mg\in \Omega^2_H(M)$ be a horizontal magnetic field.
The condition $\diff \B =0$ is equivalent to the following
\begin{equation}
\label{eq:maxwell}
	\mathrm{div}_{\mu}(-\B_2X_1+\B_1X_2)=0.
\end{equation}
where $\mu$ is the smooth measure associated with the Popp's volume.
\end{lemma}

\begin{proof}
On the one hand, recalling \eqref{eq:cartanckij}, we have that 
\begin{equation} \label{eq:dB}
\diff\B=\diff\left(\mg\right)=\left(X_1\B_2-X_2\B_1-c^1_{12}\B_1-c^2_{12}\B_2\right)\nu_1\wedge\nu_2\wedge\omega.
\end{equation}
On the other hand, by  the linearity of the divergence operator and exploiting \eqref{leibniz} and \eqref{eq:divXi}, it holds that 
\begin{equation} \label{eq:div}
\mathrm{div}_\mu(-\B_2X_1+\B_1X_2)
= -X_1\B_2+c^2_{12}\B_2+X_2\B_1+c^1_{12}\B_1.
\end{equation}

Therefore, comparing the coefficient of the volume form in \eqref{eq:dB} with the expression in \eqref{eq:div}, we deduce that $\diff\B=0$ if and only if $\mathrm{div}_\mu(-\B_2X_1+\B_1X_2)=0$.
\end{proof}

\begin{remark} 
Formula \eqref{eq:maxwell} can be seen as the analog of Maxwell's equation in classical magnetic theory in the three-dimensional Euclidean space, with $\mu$  the Euclidean volume measure and in the absence of an electric field. 
Maxwell's equations were studied in the setting of Carnot groups in terms of the Rumin complex in \cite{FT12}.
\end{remark}

\begin{remark} 
In \cite[Proposition 4.2]{CFKP23}, the authors deduce an explicit formula expressing the dependence between the coefficients of a horizontal magnetic field in the Heisenberg group written with respect to the horizontal polar frame (we follow the notation of \cref{ex:Heisenberg}).
\begin{equation}
	R=\frac{x}{r}X_1+\frac{y}{r}X_2,
	\qquad
	\Phi=-\frac{y}{r}X_1+\frac{x}{r}X_2,
\end{equation}
where $r=\sqrt{x^2+y^2}$. 
Their formula is equivalent to \eqref{eq:maxwell}.
\end{remark}

In the following result we use \eqref{eq:maxwell} to characterize the coefficient $\B_1$ of a magnetic field of the form $\B=\B_1dx\wedge\omega$ in the case of the Heisenberg group.

\begin{lemma}
Let $\B=\B_1dx\wedge\omega\in \Omega^2_H(\mathbb{H})$ be a horizontal magnetic field in the Heisenberg group. 
Then, there exists $g\in C^\infty\left(\R^2\right)$ such that
\begin{equation}
\B_1(x,y,z)=g\left(x, z-\frac{xy}{2}\right).
\end{equation} 
\end{lemma}

\begin{proof}
We start by observing that in the Heisenberg group the only non trivial commutator is given by $[X_{1},X_{2}]=X_{0}$. Then $c_{12}^{1}=c_{12}^{2}=0$, thus by \eqref{eq:divXi} and \eqref{eq:structurecoeff} it holds that $\mathrm{div}_\mu(X_2)=0$, where $\mu$ is the smooth measure associated with the Popp's volume. From \eqref{eq:maxwell}, we then obtain the following PDE:
\begin{equation}
\label{eq:PDE}
0=X_2(\B_1)=\left(\p_y+\frac{x}{2}\p_z\right)\B_1,
\end{equation} 
which can be solved by the method of characteristics.
More precisely, the function $\beta_{1}$ is constant along the flow of $X_{2}$. 
Computing explicitly the flow of $X_{2}$ we have
\begin{equation}\label{eq:moc}
\begin{cases}
\dot{x}=0\\
\dot{y}=1\\
\dot{z}=\frac{x}{2}
\end{cases}
\qquad \implies \qquad
\begin{cases}
x(t)=x(0)\\
y(t)=y(0)+s\\
z(t)=z(0)+\frac{x(0)}{2}t
\end{cases}.
\end{equation}
In particular, given an initial condition $\gamma(0)=(x_0,0,z_0)$ (defined on a surface transversal to $X_{1}$) we have that $\gamma(t)=e^{tX_{1}}\gamma(0)=(x_0,s,z_0+\frac{x_0}{2}s)$ and $$\B_1(\gamma(t))=\B_1(\gamma(0))=\B(x(t),0,z(t)-\frac{1}{2}x(t)y(t)).$$ From \eqref{eq:moc} it 
 is easy to see that both quantities $x$ and $z-\frac{1}{2}xy$ are constant along the flow.
\end{proof}

\begin{remark}
The previous consideration can be extended to $\B\in \Omega^2_H(\mathbb{H})$ horizontal magnetic field in the Heisenberg group of the form $\B=f(b_1dx+b_2dy)\wedge\omega$, with $f\in C^\infty(\mathbb{H})$ and $b_1,b_2\in\R$ such that $b_1^2+b_2^2=1$.
Then, taking into account the following change of variables
\begin{equation}
u=b_1x+b_2y,\qquad v=-b_2x+b_1y,
\end{equation}
with the same arguments, it is possible to show that there exists a suitable $g\in  C^\infty(\R^2)$ such that $f(u,v,z)= g\left(u,z-\frac{1}{2}uv\right)$.
\end{remark}

\section{Non-vanishing magnetic fields}
\label{s:nonzeromf}

In this section we study the bracket generating properties of the lifted distribution $\lift\distr$ and its abnormal curves in the case when $\beta$ is a non-vanishing horizontal magnetic field.

Recall that $\lift\distr$ is generated by $Y_1,Y_2$ in \eqref{eq:generatingfamily}, which is also an orthonormal frame for the metric. Here
$Y_0$ denotes the lift of the Reeb vector field as in  \eqref{eq:Y0}.

\begin{lemma}
\label{lem:step3}
Let $\B\in \Omega^2_H(M)$ be a horizontal magnetic field and let $A=A_1\nu_1+A_2\nu_2$ be a magnetic potential for $\B$.
For $i=1,2$, it holds that
 \begin{equation}
 \label{eq:[Yi,Y0]}
  [Y_i,Y_0]=\B_i\p_w \mod \lift \distr
 \end{equation}
In particular, if $\B$ is never vanishing,  the lifted distribution $\lift \distr$ is bracket generating  of step $3$. 
\end{lemma}

\begin{proof}
We first establish the following identity for $i=1,2$:
\begin{equation}
[Y_i,Y_0]=[X_i,X_0]+A([X_i,X_0])\p_w+\B_i\p_w.
\end{equation}
This yields \eqref{eq:[Yi,Y0]} since the distribution $\lift \distr$ is spanned by vector fields of the form $X + A(X) \partial_w$.
For $i=1,2$, using \eqref{eq:generatingfamily} and \eqref{eq:Y0}, we compute:
\begin{equation}
[Y_i,Y_0] = [X_i,X_0] + (X_i(\diff A(X_1,X_2)) - X_0(A_i)) \p_w.
\end{equation}
Define the $1$-form
$A' = A_1\nu_1 + A_2\nu_2 + \diff A(X_1,X_2) \omega$.
Then, for $i=1,2$, we can rewrite
\begin{equation}
X_i(\diff A(X_1,X_2)) - X_0(A_i) = X_i(A'(X_0)) - X_0(A'(X_i)).
\end{equation}
Recalling \eqref{eq:dH1}, we know that $\diff_H A = \diff A'$. Applying Cartan's formula \eqref{eq:cartan}, we obtain:
\begin{equation}
X_i(\diff A(X_1,X_2)) - X_0(A_i) = A'([X_i,X_0]) + \B_i.
\end{equation}
Since $\omega([X_i,X_0]) = 0$, then
$A'([X_i,X_0]) = A([X_i,X_0]),$
which verifies the desired identity \eqref{eq:[Yi,Y0]}.

Finally, since $\mathrm{span}\{Y_1,Y_2,Y_0\}=\lift \distr^{2}$ and the coordinate vector field $\p_w$ is independent of $\lift \distr^{2}$, we conclude that $[Y_i,Y_0]$ is independent of $\lift \distr^{2}$ if and only if $(\B_1,\B_2) \neq (0,0)$.
\end{proof}

Let $\B=\mg\in \Omega^2_H$ be a never vanishing horizontal magnetic field. We define the following frame for the distribution $\lift\distr$ associated with $\beta$
\begin{equation}\label{eq:F1F2}
F_1=\B_1Y_1+\B_2Y_2,  \qquad
F_2=-\B_2Y_1+\B_1Y_2.
\end{equation}

\begin{lemma}
For $F_1$, $F_2$ as in~\eqref{eq:F1F2} we have  that $[F_1,\lift\distr^2]\not\subseteq \lift\distr^2$  and $[F_2,\lift\distr^2]\subseteq \lift\distr^2.$
More precisely we have 
\begin{equation}
\label{eq:[F_1,Y_0]}
[F_1,Y_0]=\left(\B_1^2+\B_2^2\right)\p_w\mod \lift\distr,
\qquad [F_2,Y_0]\subseteq  \lift \distr.
\end{equation}
\end{lemma}

\begin{proof}
Since $F_1,F_2$ belong to $\lift\distr$, we are reduced to prove that $[F_1,Y_0]\not\subseteq\lift\distr^2$ and that $[F_2,Y_0]\subseteq\lift\distr^2$.
Recalling \eqref{eq:structurecoeff} and \eqref{eq:[Yi,Y0]}, we first compute
\begin{align}
[F_1,Y_0] &=\B_1[Y_1,Y_0]+\B_2[Y_2,Y_0] 
-X_0\left(\B_1\right)Y_1-X_0\left(\B_2\right)Y_2\\
&=\left(\B_1^2+\B_2^2\right)\p_w\mod \lift\distr.
\end{align}
Since by assumption $\left(\B_1^2+\B_2^2\right)\neq0$ and  $\lift{\distr}^2 = \mathrm{span}\left\{F_1,F_2,Y_0\right\}$, we deduce $[F_1,Y_0]\not\subseteq \lift\distr^2$.

Similarily, again using  \eqref{eq:[Yi,Y0]}, we obtain that
\begin{align}
[F_2,Y_0]&=-\B_2[Y_1,Y_0]+\B_1[Y_2,Y_0] \mod \lift\distr\\
&=-\B_2\B_1\p_w +\B_1\B_2\p_w=0 \mod \lift\distr.
\end{align}
which implies that $[F_2,Y_0] \in \lift \distr\subseteq \lift\distr^2$.
\end{proof}
\begin{remark} \label{r:rango1} 
Let  $\{\eta_1,\eta_{2},\omega\}$ be the dual frame corresponding to $\{\pi_*F_1,\pi_*F_2,X_0\}$, where $\pi:\lift M \to M$ is the canonical projection. We have that $\beta$ is ``rectified''  as follows
\begin{equation}
	\B=\|\B\|^2\eta_1\wedge\omega ,
\end{equation}
where $\|\B\|:M\to\R$ is defined as $\|\B\|=\sqrt{\B_1^2+\B_2^2}$ and it is called the \emph{magnitude} of $\B$.
\end{remark}
Now we can prove the main result of this section.
\begin{proposition}
\label{lem:abnormalsinBneq0}
Let $\B\in \Omega^2_H$ be a never vanishing horizontal magnetic field on $M$. 
A horizontal curve $\lift \gamma$ in $ \lift M$ is  an abnormal extremal trajectory if and only if its projection $\gamma$ is a horizontal curve on $M$ which is a characteristic curve of the magnetic field $\beta$ on $M$, i.e., satisfies $\iota_{\dot \gamma}\beta=0$. 
\end{proposition}

\begin{remark}
Notice that, choosing a frame $X_{1},X_{2}$ for $\distr$ as before, characteristic curves $\gamma$ of the magnetic field $\beta$ on $M$ are parallel to $-\B_2X_1+\B_1X_2$. In particular, we observe that this implies $b(\dot \gamma)=0$ in equation \eqref{eq:dbmf2}, suggesting that along abnormal extremal trajectories the magnetic field does not affect the dynamics (see Remark~\ref{r:intdb}).
\end{remark}

\begin{proof} We consider on $\lift M$ the basis of the tangent space given by $\{Y_{1},Y_{2},Y_{0},\partial_{w}\}$ where as usual
\begin{align}
Y_{i}&=X_{i}+A_{i}\partial_{w},\qquad i=1,2, \label{eq:db20} \\
Y_{0}&=X_{0}+dA(X_{1},X_{2})\partial_{w}.
\end{align}
and in the above formulas we identify $T \lift M\simeq  TM \oplus T\R$. Similarly we can consider the identification $T^{*}\lift M\simeq  T^{*}M \oplus T^{*}\R$, hence we can treat the basis of 1-forms $\{\nu_{1},\nu_{2},\nu_{0}\}$ dual to $\{X_{1},X_{2},X_{0}\}$ as 1-forms on $\lift M$, where we recall that $\nu_0=\omega$ is the contact form  on $M$. 
Following this notation, the basis $\{\nu_{1},\nu_{2},\nu_{0},\tau\}$ is dual to $\{Y_{1},Y_{2},Y_{0},\partial_{w}\}$, where we set
\begin{equation}\label{eq:tau1db}
\tau=dw-A_{1}\nu_{1}-A_{2}\nu_{2}-dA(X_{1},X_{2})\nu_{0}.
\end{equation}
Notice that $\tau$ satisfies the identities
\begin{equation}\label{eq:taudb}
\tau=dw-A',\qquad d\tau=-\beta,
\end{equation}
where $A'$ is the modified  potential in such a way that $dA'=d_{H}A=\beta$ (cf.\ \eqref{eq:dH1}).

We now use Proposition~\ref{p:sympzero}. We first compute the symplectic form $\sigma$ as a two form on $T^{*}\lift M$.  To describe $\sigma$ we need a basis of its tangent space $T(T^{*}\lift M)$. Let $p:T^{*}\lift M\to \lift M$ the canonical projection onto the base. Consider the isomorphism for $\zeta\in T^{*}\lift M$
\begin{equation}\label{eq:isomorfismodb}
T_{\zeta}(T^{*}\lift M)\simeq T_{p(\zeta)}\lift M\oplus \ker p_{*,\zeta}\simeq 
\mathrm{span}\{Y_{1},Y_{2},Y_{0},\partial_{w},	\partial_{\zeta_{1}},\partial_{\zeta_{2}},\partial_{\zeta_{0}},\partial_{\zeta_{w}}	
\}, 
\end{equation} 
where $p_{*,\zeta}$ denotes the differential of $p$ at $\zeta$, and we denote by $\zeta_{i}$, for $i=1,2,0,w$, the linear on fiber functions associated to the basis given by $\{Y_{1},Y_{2},Y_{0},\partial_{w}\}$. Namely $\zeta_{i}=\langle \zeta, Y_{i}\rangle$ for $i=0,1,2,$ and $\zeta_{w}=\langle \zeta, \partial_{w}\rangle$. 

We have that $\sigma$ is the differential of the tautological 1-form $s=\zeta_{1}\nu_{1}+\zeta_{2}\nu_{2}+\zeta_{0}\nu_{0}+\zeta_{w}\tau$ (one can see  \cite[Section~4.2]{thebook} for more details on the symplectic form in this framework). We have to compute
 $$\sigma|_{\distr^{\perp}}=ds|_{\distr^{\perp}}=d\left(s|_{\distr^{\perp}}\right),$$
since the differential commutes with the restriction. Now recall that $\distr^{\perp}=H^{-1}(0)$ and in this case the Hamiltonian $H$ writes as
$$H=\frac{1}{2}(\zeta_{1}^{2}+\zeta^{2}_{2}),$$
so that on $\distr^{\perp}$ one has $\zeta_{1}=\zeta_{2}=0$. Hence
$$\sigma|_{\distr^{\perp}}=d\left(\zeta_{0}\nu_{0}+\zeta_{w}\tau\right).$$
Since Goh conditions for rank 2 distributions are always satisfied (cf.~\cite[Section~12.4]{thebook}), using \eqref{eq:[Y1,Y2]} one can see that
\begin{equation}\label{eq:zeta3}
\zeta_{1}=\zeta_{2}=0 \qquad \text{implies} \qquad \zeta_{0}=0.
\end{equation} 
Hence characteristic curves for $\sigma$ lives in $\distr_{0}^{\perp}:=\distr^{\perp}\cap \zeta_{0}^{-1}(0)$. We are reduced to compute \begin{align*}
\sigma|_{\distr_{0}^{\perp}}=d\left(\zeta_{w}\tau\right)&=d\zeta_{w}\wedge \tau+\zeta_{w}d\tau\\
&=d\zeta_{w}\wedge \tau-\zeta_{w}\beta,
\end{align*}
where in the last equality we used \eqref{eq:taudb}.
Notice that $\distr_{0}^{\perp}$ is a 5-dimensional manifold in $T^{*}M$. By considering the isomorphism \eqref{eq:isomorfismodb},
we can write
$$T\distr_{0}^{\perp}=\mathrm{span}\{Y_{1},Y_{2},Y_{0},\partial_{w},\partial_{\zeta_{w}}\}.$$
We write more explicitly
\begin{align*}
\sigma|_{\distr_{0}^{\perp}}=d\zeta_{w}\wedge \tau-\zeta_{w}(\beta_{1}\nu_{1}\wedge \nu_{0}+\beta_{2}\nu_{2}\wedge \nu_{0}),
\end{align*}
and we recall that (a) the distribution $\lift \distr=\mathrm{span}\{Y_{1},Y_{2}\}$ is contained in $\ker \tau$ by \eqref{eq:db20} and \eqref{eq:tau1db}, (b) $\zeta_{w}\neq 0$ due to identities \eqref{eq:zeta3} and the fact that the lift of the abnormal trajectory is never vanishing by Proposition~\ref{p:sympzero}. It is then simple to observe that the kernel of $\sigma|_{\distr_{0}^{\perp}}$ coincides with the vector $-\beta_{2}Y_{1}+\beta_{1}Y_{2}$, which projects on the kernel of the magnetic field $\beta$ as a 2-form.
\end{proof}

\begin{example}
Recall that in \cref{ex:fromHtoEintro} we built the classical Engel group equipping the Heisenberg group $\mathbb{H}$ with the magnetic field $\B=dx\wedge\omega$.
Since $\B$ is non-vanishing, by \cref{lem:abnormalsinBneq0} it follows that the abnormal curves in the lifted space are tangent to $Y_2$ and a parametrization of an abnormal curve $\gamma:[0,T]\to\lift M$ with $\gamma(0)=(x_0,y_0,z_0,w_0)\in\mathbb{E}$ is 
\begin{equation}
 \gamma(t)=e^{tY_2}(x_0,y_0,z_0,w_0)=\left(x_0,y_0+t,z_0-\frac{x_0}2t,w_0-\frac{x_0^2}{2}t\right).
\end{equation}
This is in accordance with the classical characterization of abnormal curves in the Engel group. 
\end{example}

\section{Vanishing magnetic fields}
\label{s:zeromf}
In \cref{lem:step3} we proved that if $\B$ is a non-vanishing magnetic field, then the lifted distribution $\lift\distr$ is bracket generating of step 3 and \cref{lem:abnormalsinBneq0} characterizes the corresponding abnormal curves. 
 In this section, we investigate the case when $\B$ vanishes in some region.
 
We denote with $\mathcal{Z}\subset M$ the zero locus of $\B$, i.e., the set where $\beta_{1}=\beta_{2}=0$, and we study the lifted \sr structure on $\lift M$ at $\mathcal{Z}\times \R$.
The first result we present can be seen as an extension of  \cref{lem:step3}.

\begin{proposition}
\label{lem:step}
Let $p\in \mathcal{Z}$, $k\in \mathbb{N}, k\geq 1$.
The step of the distribution $\lift\distr$ at $\{p\}\times \R \subset \lift M$ is equal to $k+3$ if and only if the following conditions are satisfied
\begin{itemize}
\item[(a)] There exists a choice of indexes $(i_0,\ldots,i_k)\in\{1,2\}^{k+1}$ such that 
\begin{equation}
X_{i_k}\ldots X_{i_{1}}\B_{i_0}(p)\neq 0;
\end{equation}
\item[(b)] For every choice of indexes $(i_0,\ldots,i_{j})\in\{1,2\}^{j+1}$ with $j=0,\ldots,k-1$ it holds that 
\begin{equation}
X_{i_{j}}\ldots X_{i_{1}}\B_{i_{0}}(p)= 0.
\end{equation}
\end{itemize}
\end{proposition}

\begin{proof}
Let $A$ be a horizontal potential for $\B$, i.e., such that $\B=\diff_H A$.
Let $Y_1,Y_2$ be the generating family for $\lift\distr$ as in \eqref{eq:generatingfamily},
and let $Y_0$ be as in \eqref{eq:Y0}.
It holds that 
\begin{equation}
	\lift\distr^2=\mathrm{span}\left\{Y_1,Y_2,Y_0\right\}.
\end{equation}
Moreover, for $i\in\{1,2\}$, identity \eqref{eq:[Yi,Y0]} gives
\begin{equation}
[Y_{i},Y_0]=\B_{i} \p_w \mod \lift\distr.
\end{equation}
Since $\p_w$ is independent from $\lift\distr^2$, we have that 
\begin{equation}
	\lift\distr^3=\mathrm{span}\left\{Y_1,Y_2,Y_0,\B_1\p_w,\B_2\p_w\right\}.
\end{equation}
We deduce that $\lift \distr^3=\lift\distr^2$ at $ \mathcal{Z}\times \R$ since $\beta_{1}=\beta_{2}=0$ on $\mathcal{Z}$. 
Recursively, one finds that, for $k\in \mathbb{N}$, $k\geq 1$
\begin{equation}
	\lift\distr^{k+3}=\mathrm{span}\left\{ Y_1,Y_2,Y_0, 
							\left(X_{i_j}\ldots X_{i_{1}}\B_{i_0}\right) \p_w 
							\mid j=0,\ldots,k, \ (i_0,\ldots,i_j)\in\{1,2\}^{j+1}
							\right\}.
\end{equation} 
It follows that $\lift\distr^{k+3}=\lift\distr^2$ on $\mathcal Z\times \R$ if and only if $X_{i_j}\ldots X_{i_{1}}\B_{i_0}(p)=0$ for all $j=0,\ldots,k$ and $ (i_0,\ldots,i_j)\in\{1,2\}^{j+1}$. This concludes the proof since $\lift\distr^{k+3}\neq \lift\distr^2$ implies $\lift\distr^{k+3}=T\lift M$.
\end{proof}

\begin{remark}
We stress that if there exists a point $p\in  \mathcal{Z}$ such that any horizontal derivation of any order of $\B_1$ and $\B_2$ vanishes at such a point, then the lifted distribution $\lift \distr$ is not bracket generating at $\{p\}\times \R$. Moreover, since any derivation can be written as a combination of horizontal ones, this is equivalent to require that any derivation of any order vanishes at $p$.

Finally, let us stress that the statement of Proposition~\ref{lem:step} is actually independent of the chosen (orthonormal) frame for $\distr$ in $M$. 
\end{remark}

Under the assumption that the magnetic field generates a sub-Riemannian structure, then every horizontal curve in the lift of the zero locus is an abnormal extremal trajectory.

\begin{proposition}\label{p:dbp}
Let  $\lift \gamma$ be a horizontal curve in $\lift M$ contained in  $\mathcal{Z}\times \R$. Then $\lift \gamma$ is an abnormal extremal trajectory for the lifted \sr structure on $\lift{M}$.
\end{proposition}

\begin{proof} The same argument given in the proof of Proposition~\ref{lem:abnormalsinBneq0} yields for the restriction of the symplectic form  $\distr_{0}^{\perp}:=\distr^{\perp}\cap \zeta_{0}^{-1}(0)$ 
\begin{align*}
\sigma|_{\distr_{0}^{\perp}}=d\left(\zeta_{w}\tau\right)&=d\zeta_{w}\wedge \tau+\zeta_{w}d\tau\\
&=d\zeta_{w}\wedge \tau-\zeta_{w}\beta.
\end{align*}
Restricting to curves on $T^{*}\lift M$ which projects on $\mathcal{Z}\times \R$, then $\sigma|_{\distr_{0}^{\perp}}=d\zeta_{w}\wedge \tau$.
Denoting points on $T^{*}\lift M$ as tuples $(\lift p,\zeta_{1},\zeta_{2},\zeta_{0},\zeta_{w})$ then the lift $(\lift \gamma(t),0,0,0,c)$ of the curve $\lift \gamma(t)$ on $\mathcal{Z}\times \R\subset \lift M$, for some $c\neq 0$, is a characteristic curve of $\sigma|_{\distr_{0}^{\perp}}$. Hence $\lift \gamma$ is an abnormal extremal trajectory.
\end{proof}

\begin{remark} Abnormal extremal trajectory can also be a concatenation of a curve in the zero locus of the magnetic field and a characteristic curve of $\beta$ where $\beta$ is non-vanishing. An explicit example of such an abnormal extremal trajectory is provided in Example~\ref{ex:crossing1}.
\end{remark}

In the following we analyze the step of the lifted distribution at $ \mathcal{Z} \times \R \subset \lift M$ in relation to the rank of the map $(\diff \B_1,\diff \B_2):TM\to\R^2$. Namely, we consider $p\in  \mathcal{Z}$ and the rank of the linear map $\restr{(\diff \B_1,\diff \B_2)}{p}:T_pM\to\R^2$.

\begin{proposition}
\label{lem:step-rank}
Let $\B=\mg$ be a horizontal magnetic field. Fix $p\in  \mathcal{Z}$. 
\begin{enumerate}[(i)]
\item \label{item:step2} If $\mathrm{rank}\restr{\left(\diff \B_1,\diff \B_2\right)}{p}=2$, then the step of $\lift \distr$ at $\{ p \}\times \R \subset \lift M$ is $4$;
\item \label{item:step1} If $\mathrm{rank}\restr{\left(\diff \B_1,\diff \B_2\right)}{p}=1$, then the step of $\lift \distr$ at $\{ p \}\times \R \subset \lift M$ is $4$ or $5$;
\item \label{item:step0} If $\mathrm{rank}\restr{\left(\diff \B_1,\diff \B_2\right)}{p}=0$, then the step of $\lift \distr$ at $\{ p \}\times \R \subset \lift M$ is $\geq 5$.
\end{enumerate}
\end{proposition}

\begin{proof}
First observe that the matrix representing  $\left(\diff \B_1,\diff \B_2\right)$ written with respect to the frame $X_1,X_2,X_0$, is given by
\begin{equation}
\label{eq:matrixDB}
 \left( \begin{array}{ccc}
 X_1\B_1 &  X_2\B_1 &  X_0\B_1 \\
  X_1\B_2 &  X_2\B_2 &  X_0\B_2 
 \end{array} \right).
\end{equation}

\ref{item:step2}.
First, assume that $\mathrm{rank}\left(\diff \B_1,\diff \B_2\right)=2$ at $p$. In this case at least two columns of \eqref{eq:matrixDB} evaluated at $p$ are independent and not vanishing.  
Hence, there exist $i,j\in\{1,2\}$ such that $X_i\B_j(p) \neq 0$ and the conclusion holds by \cref{lem:step}.

\ref{item:step1}.
Let us now assume $\mathrm{rank}\left(\diff \B_1,\diff \B_2\right)=1$ at $p$, then at least one column in \eqref{eq:matrixDB} evaluated at $p$ is not identically zero. 
If this happens for one of the first two columns, we can argue as in the proof of (i), and the step of $\lift \distr$ at ${\{p\}\times \R}$ is $4$.
If the first two columns are both identically zero, any first order horizontal derivation of $\B_1$ and $\B_2$ is vanishing at $p$, and there exists $i\in \{1,2\}$  such that $X_0\B_i(p)\neq 0$. 
Recalling \eqref{eq:structurecoeff}, we have that
\begin{equation}
[X_1,X_2]\B_i(p)=c^1_{12}X_1\B_i(p)+c^2_{12}X_2\B_i(p)+X_0\B_i(p)=X_0\B_i(p).
\end{equation}
Since in addition $[X_1,X_2]\B_i=X_1X_2\B_i-X_2X_1\B_i$, we conclude that at least a second order derivative does not vanish at $p$. 
The step at $\lift \distr$ at ${\{p\}\times \R}$ is then $5$.

\ref{item:step0}.
Assume $\mathrm{rank}\left(\diff \B_1,\diff \B_2\right)=0$ at $p$, then the matrix \eqref{eq:matrixDB} is identically zero when evaluated at $p$. Hence, for all $i,j\in\{1,2\}$ we have $X_i\B_j(p)= 0$. \cref{lem:step} implies that the step of $\lift\distr$ at $\{p\}\times\R$ is at least $5$.
\end{proof}

We illustrate the content of \cref{lem:step-rank} on some examples of horizontal magnetic fields $\B$ on the Heisenberg group $\mathbb{H}$ (we refer to \cref{ex:fromHtoEintro} for the notation in $\mathbb{H}$).

\begin{example}[The rank-2 case] \label{ex:rank2}
If $\mathrm{rank}\left(\diff\B_1,\diff\B_2\right)=2$ at a point $p\in \mathcal{Z}$, then $\mathcal{Z}$ is (locally) a curve regularly defined by the map $\left(\B_1,\B_2\right):M\to \R^2$.
We highlight that the curve being horizontal or not does not affect the step of $\lift \distr$ at $\mathcal{Z}\times \R$, which is always equal to $4$.

Two representative examples on $\mathbb{H}$ are given by the following choices:
\begin{equation}
\B=x\,dx\wedge\omega+y\,dy\wedge\omega, \qquad \B'=\left(z-\frac{xy}{2}\right)dx\wedge\omega+y\,dy\wedge\omega.
\end{equation} 
The zero locus $\mathcal{Z}$ of  $\B$ is the non-horizontal curve $\{x=y=0\}$, the integral curve of $X_0$. 
The zero locus $\mathcal{Z}'$ of $\B'$ is the horizontal curve $\{z=y=0\}$, the integral curve of $X_1$.
\end{example}

\begin{example}[Step variation in the rank-1 case]
\label{ex:rank1}
We show that if $\mathrm{rank}\left(\diff\B_1,\diff\B_2\right)=1$ on $\mathcal{Z}$, the step of $\lift\distr$ is not necessarily constant on $\mathcal{Z}\times\R$. 
Let us consider on $\mathbb{H}$
 \begin{equation}
 	\B=4z\,dx\wedge\omega+x^2\,dy\wedge\omega.
 \end{equation}
 The zero locus is the curve $\mathcal{Z}=\{(0,y,0)\mid y\in \R\}$, so that, writing  $ \left( \diff\B_1, \diff\B_2\right)$ with respect to the frame $X_1,X_2,X_0$ given in \cref{ex:Heisenberg}, we have
  \begin{equation}
 \left( \diff\B_1, \diff\B_2\right)=
 \left(\begin{array}{ccc}
 -2y & 2x &4 \\
 2x & 0 &0
 \end{array}\right),
 \qquad
  \restr{\left( \diff\B_1, \diff\B_2\right)}{\mathcal{Z}}=
 \left(\begin{array}{ccc}
 -2y & 0 & 4\\
 0 & 0 & 0
 \end{array}\right).
 \end{equation} 
Hence, $\mathrm{rank}\left(\diff \B_1,\diff \B_2\right)=1$ at $p\in \mathcal{Z}$.
We have that  the step of $\lift \distr$ is $5$ at $(0,0,0)\times \R$ and $4$ at points $(0,y,0)\times \R$ with $y\neq 0$. 
Indeed, if $y\neq 0$, $X_1\B_1(0,y,0)=2y\neq 0$ implying that the step is $4$ by Proposition~\ref{lem:step}. 
For $y=0$, we have that $X_1\B_j(0,0,0)=X_2\B_j(0,0,0)=0$ and $X_2X_1\B_1(0,0,0)=-2$, yielding step 5.
\end{example}

\begin{example}[The rank-0 case]
\label{exa:r0}
We exhibit an example where $\mathrm{rank}{\left(\diff \B_1,\diff \B_2\right)}=0$ in $\mathcal{Z}$, for which the \sr structure have arbitrary step $\geq 5$ at $\mathcal{Z}\times\R$.

Fix $n\geq 2$ and consider $\B=\frac{x^n}{n!}\,dx\wedge\omega$ on the Heisenberg group $\mathbb{H}$. 
The zero locus of $\beta$ is given by the surface $\mathcal{Z}=\{x=0\}$. Since $n\geq 2$, it is easy to check that the rank is zero at $\mathcal{Z}$.

Let $p\in\mathcal{Z}$, we have that $X_{i_k}\ldots X_{i_1}\B_1(p)=X_{i_k}\ldots X_{i_1}\B_2(p)=0$ for any $0\leq k \leq n-1$ and choice of indices $(i_k,\ldots,i_1)\in \{1,2\}^{k}$. 
Being $X_1^n\B_1=\p_x^n \frac{x^n}{n!}(p)=1$, we conclude by Proposition~\ref{lem:step} that the step of $\lift\distr$ at $\mathcal{Z}\times\R$ is $n+3$.

\end{example}

\subsection{Characteristic points in the rank-1 case}

If $\restr{\left(\diff \B_1,\diff \B_2\right)}{p}:T_pM\to \R^2$ has rank 1 at $p\in \mathcal{Z}$, there exists $i\in \{1,2\}$ such that $\diff \B_i|_{p}\neq 0$.
Therefore, up to restricting the domain  to an open neighborhood of $p$, it is well defined a regular surface $\Sigma=\B_{i}^{-1}(0)$ that contains the zero locus $\mathcal{Z}$ of the magnetic field $\beta$ around $p$.

Recall that, given $\Sigma$ a surface in $M$, we say that a point $p\in \Sigma$ is \emph{characteristic} if  $\distr_p=T_p\Sigma$.
Denote with $\mathrm{Char}(\Sigma)$ the set of characteristic points in $\Sigma$.
 The following result characterizes the step of $\lift \distr$ in geometric terms.

\begin{proposition}\label{prop:rank1}
Let $\B=\mg\in\Omega^2_H(M)$ be a horizontal magnetic field such that $\mathrm{rank}\restr{\left(\diff\B_1,\diff\B_2\right)}{p}=1$ with $p\in \mathcal{Z}$.
Let $i\in \{1,2\}$ such that $\restr{\diff \B_i}{p}\neq 0$
and let $\Sigma$ be the smooth surface locally defined by $\B_i$.
The following are equivalent
\begin{enumerate}[(i)]
	\item \label{item:surface2} $p$ is a characteristic point in $\Sigma$;
	\item \label{item:surface3} $\lift\distr$ has step $5$ at $\{p\}\times \R$.
\end{enumerate}

\end{proposition}
\begin{remark}
Using the notation of Proposition~\ref{prop:rank1} above, we observe that a point $p\in \Sigma$ is characteristic if and only if $\lift\distr \subset T\Sigma\times \R$ at $\{p\}\times \R$. The inclusion here is crucial since in $\lift M$ the lifted distribution $\lift \distr$ is two-dimensional while $\Sigma\times \R$ is three-dimensional.

\end{remark}

\begin{proof}
Without loss of generality, we can assume $i=1$. The rank assumption implies that $\restr{\diff \B_2}{p} = c \restr{\diff \B_1}{p}$ for some $c\in \R$. Consequently,  $\ker \diff \B_1|_{p} \subseteq \ker \diff \B_2|_{p}$.

A point $p \in \Sigma=\beta_{1}^{-1}(0)$ is characteristic if and only if $X_1 \B_1(p) = X_2 \B_1(p) = 0$. Due to the inclusion of the kernels, $p$ is characteristic in $\Sigma$ if and only if $X_i \B_j(p) = 0$ for all  $i, j \in \{1,2\}$.

By \cref{lem:step}, the step of $\lift \distr$ at  $\{p\} \times \R$ is larger or equal than $5$.  
By \cref{lem:step-rank} the assumption $\mathrm{rank}\left(\diff \B_1, \diff \B_2\right) = 1$ at a point $p \in \mathcal{Z}$ implies that the step of $\lift \distr$ must be either $4$ or $5$. We conclude that the step is precisely $5$.
\end{proof}

\begin{remark}
Let us consider again the horizontal magnetic field \( \B = 4z \, dx \wedge \omega + x^2 \, dy \wedge \omega \) on the Heisenberg group $\mathbb{H}$, as studied in \cref{ex:rank1}. The magnetic field \( \B \) satisfies the hypotheses of \cref{prop:rank1}, with the surface \( \Sigma \) being regularly defined by the zero level set of \( \B_1 = z \).

Since \( (0,0,0) \in \mathcal{Z} \) is the unique characteristic point of \( \Sigma=\{z=0\} \), the conclusions obtained in \cref{ex:rank1} regarding the step of \( \lift \distr \) could also be obtained by applying \cref{prop:rank1}.
\end{remark}

\begin{remark}
The two implications \ref{item:surface3} $\Leftrightarrow $  \ref{item:surface2} in \cref{prop:rank1} do not hold if ${\mathrm{rank}\left(\diff\B_1,\diff\B_2\right)}\neq 1$ at $p\in \mathcal{Z}$. We illustrate this fact by showing two examples of horizontal magnetic fields $\B$ on the Heisenberg group $\mathbb{H}$ (for the notations we refer to \cref{ex:fromHtoEintro}).

The fact that \ref{item:surface3} $\not \Rightarrow $  \ref{item:surface2} can be realised by considering 
$\B=\frac{x^2}{2}\,dx\wedge\omega$.
The zero locus $\mathcal{Z}=\{x=0\}$ is a surface that does not contain characteristic points, but the step of ${\lift\distr}$ is $5$ at ${\mathcal{Z}\times \R}$, as showed in \cref{exa:r0}.
Here $\mathrm{rank} \restr{\left(\diff\B_1,\diff\B_2\right)}{\mathcal{Z}}=0$.

To observe that \ref{item:surface2} $\not \Rightarrow $  \ref{item:surface3} we consider
$\B=\left(z-\frac{xy}{2}\right)dx\wedge\omega + y\,dy\wedge\omega$ in the Heisenberg group.
The zero locus is the horizontal curve $\mathcal{Z}=\{z=y=0\}$, which is contained in the set of characteristic points of the surface $S_1=\{\B_1=0\}$.
However, since $\mathrm{rank} \restr{\left(\diff\B_1,\diff\B_2\right)}{\mathcal{Z}}=2$ the step of $\lift \distr$ at $\mathcal{Z}\times \R$ is $4$ (cf.\ \cref{lem:step-rank}).
\end{remark}

\section{Examples and final remarks}
\label{s:examples}
We conclude by considering two particular classes of examples.
\subsection{A class of magnetic fields having a given surface as zero locus} 
In the following, given a regular surface $\Sigma\subset M$, we consider a class of horizontal magnetic fields whose zero locus coincides with $\Sigma$.  This class contains in particular Example~\ref{exa:r0}.

\begin{lemma}
\label{lem:example}
Let $n\in\mathbb{N}$, $n\geq 1$ and $f\in C^{\infty}(M)$ be a submersion. Consider a horizontal magnetic field of the form
\begin{equation}
\B=\frac{f^n}{n!}(b_1\nu_1+b_2\nu_2)\wedge\omega\in \Omega^2_H,
\end{equation}
where $b_1,b_2\in C^\infty(M)$ with  $b_1^2+b_2^2 \neq 0.$
Then, the vector field $-b_2X_1+b_1X_2$ is tangent to the zero locus of $\B$ given by $\Sigma=f^{-1}(0)$.
Moreover, $\mathrm{Char}(\Sigma)$ is a 1-dimensional smooth manifold.
\end{lemma}

\begin{proof}
We start by observing that, up to reabsorbing a non zero factor in the function $f$ (thus without changing the zero level set), we can assume that $b_1^2+b_2^2=1$.
Let us consider the orthonormal frame for $\distr$:
\begin{equation}
V_1= b_1X_1+b_2X_2, \qquad V_2 = -b_2X_1+b_1X_2,
\end{equation}
with its corresponding dual basis:
\begin{equation}
\eta_1= b_1\nu_1+b_2\nu_2, \qquad \eta_2=-b_2\nu_1+b_1\nu_2.
\end{equation}
In this basis, we have
\begin{equation}
\B=\frac{f^n}{n!}\eta_1\wedge\omega.
\end{equation}
{\bf Claim (i)}. For every $p\in \Sigma$ we have $V_2f(p)=0$, i.e., the vector field $V_{2}$ is tangent to $\Sigma=f^{-1}(0)$.

\smallskip
Let $\mu$ be the measure associated with the Popp's volume.
Recalling the closure condition \eqref{eq:maxwell}, and applying the Leibniz's rule for the divergence operator in \eqref{leibniz}, we obtain that  
\begin{equation}
0=\mathrm{div}_\mu\left(\frac{f^n}{n!}V_2\right)=\frac{f^{n-1}}{(n-1)!}\left(V_2f+\mathrm{div}_\mu(V_2)\frac{f}{n}\right).
\end{equation}
Therefore, on $M\setminus \Sigma$ (i.e., the set where $f\neq 0$) we have
\begin{equation}
\label{eq:closure}
V_2f+\mathrm{div}_\mu(V_2)\frac{f}{n}=0.
\end{equation} 
Since $\Sigma=f^{-1}(0)$ is a closed set with empty interior, identity \eqref{eq:closure} holds by continuity also on $\Sigma$ and since on $
\Sigma$ we have $f=0$, the Claim (i) is true.

\smallskip
To prove that $\mathrm{Char}(\Sigma)$ is a 1-dimensional smooth manifold, first notice that $\mathrm{Char}(\Sigma)\subset \Sigma$ is the zero level-set of the map $\Phi:M\to \R^2$ defined as $\Phi=(f,V_1f)$. 
Indeed, we have that $p\in \Sigma$ is characteristic if and only if  $V_1f(p)=V_2f(p)=0$ but  $p\in \Sigma$ implies that $V_2f(p)=0$ by \eqref{eq:closure}.

In order to conclude, we show that  for every $p\in \Phi^{-1}(0)$ it holds $\mathrm{rank}\, \diff _p\Phi=2$.
Let us consider the matrix representing $\diff \Phi$ with respect to the frame $V_1,V_2, X_0$ for $TM$:
\begin{equation}\label{eq:db130}
\diff \Phi =
\left( \begin{array}{ccc}
V_1f & V_2f &X_0f \\
V_1V_1f &V_2 V_1f & X_0 V_1f
\end{array} \right).
\end{equation}
{\bf Claim (ii):} Let $p\in \Phi^{-1}(0)$. We have that $V_2V_1f(p)= -X_0f(p)$.

\smallskip
To prove this claim, we first observe that,
differentiating \eqref{eq:closure}, one has
\begin{equation}
\label{eq:V_1V_2f}
V_1 V_2 f =  V_1 \left(-\frac{f}{n}\mathrm{div}_{\mu}(V_2)\right)=-\frac{V_2f}{n}\mathrm{div}_{\mu}(V_2)-\frac{f}{n}V_2\left(\mathrm{div}_{\mu}(V_2)\right).
\end{equation}
For $p\in  \Phi^{-1}(0)$, where $f=V_{2}f=0$, we deduce that $V_1V_2f(p)=0$.

On the other hand, being $V_1,V_2$ an orthonormal frame for $\distr$ with $\diff\omega(V_1,V_2)=-1$, by \eqref{eq:structurecoeff} there exist $a^1_{12},a^2_{12}\in C^\infty(M)$ such that 
\begin{equation}
[V_1,V_2]=a^1_{12}V_1+a^2_{12}V_2+X_0.
\end{equation}
Hence, we obtain that $[V_1,V_2]f(p)=X_0f(p)$ for all $p\in   \Phi^{-1}(0)$. 
Being $V_1V_2f(p)=0$, we deduce
\begin{equation}
\label{eq:commutatorscharS}
V_2V_1f(p)= -X_0f(p),
\end{equation} 
which proves Claim (ii). Replacing \eqref{eq:commutatorscharS} and $V_1V_2f(p)=0$ in  \eqref{eq:db130}, we have that
\begin{equation}
\diff_p \Phi =
\left( \begin{array}{ccc}
0 & 0 &X_0f(p) \\
V_1V_1f(p) &   -X_0f(p)
& X_0 V_1f(p)
\end{array} \right).
\end{equation}
Since, by assumption, $f$ is a submersion and $V_1f(p)=V_2f(p)=0$, then necessarily $X_0f(p)\neq 0$. Hence we conclude that $\mathrm{rank}\, \diff _p \Phi =2$, and this concludes the proof.
\end{proof}

We now compute the step of the lifted distribution $\lift\distr$ associated with the class of examples given in \cref{lem:example}.

\begin{lemma}
\label{lem:examplestep}
Under the  assumptions of \cref{lem:example}, it holds that 
\begin{enumerate}[(i)]
\item \label{item:3} The step of $\lift\distr$ at $(M \setminus \Sigma)\times \R$  is $3$;
\item \label{item:n+3} The step of $\lift\distr$ at $(\Sigma \setminus{ \mathrm{Char}(\Sigma)})\times\R$ is $n+3$;
\item \label{item:2n+3} The step of $\lift\distr$ at $ \mathrm{Char}(\Sigma)\times \R$ is $2n+3$.
\end{enumerate}
\end{lemma}

\begin{proof}
We apply \cref{lem:step} with respect to the orthonormal frame $V_1,V_2$ for $\distr$ defined in the proof of \cref{lem:example}.
Notice that with this choice, we have $\B_1=\frac{1}{n!}f^{n}$ and $\B_2=0$. 
Item \ref{item:3} follows immediately since on $M \setminus \Sigma$ the magnetic field is non-vanishing.

Let $k\in \mathbb{N}$, $k\geq 1$ 
and  $({i_1},\ldots,{i_k})\in\{1,2\}^k$. Then

\begin{equation}
\label{eq:differentiation}
V_{i_k}\ldots V_{i_1}\B_1 
	= \frac{1}{n!}
\sum_	{\sigma\in D_{k,n}}		
	f^{e}
\prod_{j \colon r_{j}=1}
	V_{i_{\sigma^{-1}(j)}}f
\prod_{j \colon r_{j}>1}		
	V_{i_{h_{r_{j}}}}\ldots V_{i_{h_1}}f,
\end{equation}
where with
$D_{k,n}$ we denote the set of maps $\sigma:\{1\ldots,k\}\to \{1\ldots,n\}$
and for any $\sigma\in D_{k,n}$ and  $j\in \{1,\ldots,n\}$ we set $\sigma^{-1}(j)=\{h_1<\ldots<h_{r_{j}}\}$,  
	$r_{j}=\#\sigma^{-1}(j)$
	and $e=\#\{j\colon r_{j}=0\}$.

Item \ref{item:n+3}.\
We now focus on $\Sigma\times \R$.
For $ 1\leq k\leq n-1$ we have necessarily  $e\neq 0$ for every $\sigma$. 
Hence $f^e$ is vanishing on $\Sigma$ and we obtain that $\restr{V_{i_k}\ldots V_{i_1}\B_1}{\Sigma}=0$.

Let $k=n$.  We observe that $e=0$ if and only if $\sigma$ is a permutation of $\{1\ldots,n\}$. 
Therefore,  the sum in \eqref{eq:differentiation} evaluated on $\Sigma$ reduces to  the set of the permutations $P_{n}\subsetneq D_{n,n}$ of $\{1,\ldots, n\}$. 
And, recalling that $\sigma\in P_{n}$ is bijective, \eqref{eq:differentiation} on $\Sigma$ rewrites as
\begin{equation}
\label{eq:derivatives_n}
\restr{V_{i_n}\ldots V_{i_1}\B_1}{\Sigma}
=\frac{1}{n!}\sum_{\sigma\in P_{n}}
	\prod_{j =1}^n
		{V_{i_{\sigma^{-1}(j)}}f}|_{\Sigma}
=\prod_{j=1}^n{V_{i_j}f}|_{\Sigma},
\end{equation} 
Since by \cref{lem:example} one has $\restr{V_2f}{\Sigma}=0$, then it holds that 
\begin{equation}
\begin{cases}
V_1f(p)=0, & p\in \mathrm{Char}(\Sigma),\\
V_1f(p)\neq0, & p\in \Sigma\setminus \mathrm{Char}(\Sigma).
\end{cases}
\end{equation}
being $V_1,V_2$ a frame for $\distr$.
 Choosing $({i_1},\ldots,{i_n})=(1,\ldots,1)$ in \eqref{eq:derivatives_n} we get $V_1^n\B_1(p)\neq 0$, hence by \cref{lem:step} we have that $\lift\distr$ has step $n+3$ at $(\Sigma\setminus \mathrm{Char}(\Sigma))\times \R$.

Item \ref{item:2n+3}.\
We are left to compute the step of $\lift\distr $ at ${\mathrm{Char}(\Sigma)}\times\R$. 
Let now $n+1\leq k \leq 2n-1$ and $(i_1,\ldots,i_k)\in\{1,2\}^k$. 
Evaluating \eqref{eq:differentiation} at $p\in\mathrm{Char}(\Sigma)$, again any term associated with a $\sigma$ having $e\neq 0$ is vanishing because $f(p)=0$. 
Similarly, any term associated with a $\sigma$ having at least one $r_{j}=1$ vanishes since for $p\in \mathrm{Char}(\Sigma)$ one has $V_{1}f(p)=V_{2}f(p)=0$. 
Finally, it is not possible to find $\sigma\in D_{k,n}$ with $r_{j}\geq 2$ for all $j\in\{1,\ldots, n\}$ in \eqref{eq:differentiation}, due to $k< 2n$.
Thus, we have proved that given $(i_1,\ldots,i_k)\in\{1,2\}^k$ we have
\begin{equation}
\restr{ V_{i_k}\ldots V_{i_1}\B_1}{\mathrm{Char}(\Sigma)}=0, \qquad \text{for}\  n+1\leq k\leq 2n-1.
\end{equation}
This proves that the step of $\lift\distr $ is at least $2n+3$ at points of ${\mathrm{Char}(\Sigma)}\times\R$.
To conclude, we exhibit the following $2n$-th order horizontal derivative which is non-vanishing:
\begin{equation}
\restr{(V_2V_1)^{n}\B_1}{{\mathrm{Char}(\Sigma)}}=\left(V_2 V_1f\right)^n=\left(-X_0f\right)^n\neq 0,
\end{equation}
where in the second equality we used identity \eqref{eq:commutatorscharS}.
\end{proof}

\begin{remark}
Adapting the above argument, one can actually extend the conclusions (locally in a neighborhood of the zero level set) to horizontal magnetic fields of the form
\begin{equation}
\B=g(f)(b_1\nu_1+b_2\nu_2)\wedge\omega\in \Omega^2_H,
\end{equation}
where $b_1,b_2,f\in C^\infty(M)$ with $b_1^2+b_2^2=1$, and $g\in C^\infty(\R)$ satisfies $g(0)=0$ and  for $t\to 0$ 
\begin{equation}
	g(t)=\frac{t^n}{n!}(C+o(1))
\end{equation}  
with $n\geq 1$ and $C\neq 0$. 
This permits to analyze, for instance, magnetic fields in the Heisenberg group of the form $P(x)dx\wedge\omega$ where $P$ is a polynomial in $x$.
\end{remark}

\subsection{Abnormals extremal trajectories and zero locus of the magnetic field}

We consider different situations in which abnormal extremal trajectories do enter or do not enter the lifted set $\mathcal{Z}\times \R$.
Since, by Theorem~\ref{intro:nonvanishing}, abnormal extremal trajectories on $\lift M$ project onto $M$ on characteristic curves of $\beta$, we can restrict our analysis to checking on the manifold $M$ whether or not characteristic curves of $\beta$ enter into its zero set $\mathcal{Z}$.

All examples we consider here are built on $M=\mathbb{H}$ the Heisenberg group.
\begin{example}
The first example (cf.\ Example~\ref{ex:rank2}) is a horizontal magnetic field for which characteristic curves of the magnetic field do not enter the zero set $\mathcal{Z}\times \R$. 
Let us consider 
\begin{equation}
\B=x\,dx\wedge\omega+y\,dy\wedge\omega.
\end{equation}
We have that $\mathcal{Z}=\{x=y=0\}$ is the the $z$-axis. 
In the complement of $\mathcal{Z}$, characteristic curves for $\beta$ are tangent to the vector field
\begin{equation}
-yX_1+xX_2=-y\p_x+x\p_y+\frac{1}{2}(x^{2}+y^{2})\p_z.
\end{equation}

It is easy to see that characteristic curves for $\B$ are spirals contained in cylinders around the $z$-axis, never crossing $\mathcal{Z}$.
Since $\mathcal{Z}$ is transversal to $\distr$ at every point, we stress that abnormal curves in the lift of $\mathcal{Z}$ are constant curves.
\end{example}
\begin{example}
\label{ex:crossing0}
The second example we present is a horizontal magnetic field for which there exist characteristic curves of the magnetic field crossing the zero locus $\mathcal{Z}$. This produces a non-smooth abnormal extremal trajectory on $\lift M$.
Let us consider in $\mathbb{H}$
\begin{equation}
\B=y\,dx\wedge\omega+x\,dy\wedge\omega.
\end{equation}
We have that $\mathcal{Z}=\{x=y=0\}$ is again the the $z$-axis. 
In the complement of $\mathcal{Z}$, characteristic curves for $\beta$ are tangent to the vector field
\begin{equation}
-xX_1+yX_2=-x\p_x+y\p_y+xy\partial_{z}.
\end{equation}
It is easy to see that the curve $\gamma:[0,T]\to \mathbb{H}$ defined as
\begin{equation}
\gamma (t) = \begin{cases}
\gamma_{-}(t)=(\frac T 2 - t, 0,0), & t\in\left[0,\frac{T}{2}\right]\\
\gamma_{+}(t)=(0,t- \frac T 2,0),  & t\in\left(\frac{T}{2}, T\right]
\end{cases}
\end{equation}
is a characteristic curve for the magnetic field for $t\neq \frac{T}{2}$ crossing $\mathcal{Z}$ when $t= \frac{T}{2}$.
Moreover, since $\mathcal{Z}$ is transversal to $\distr$ at every point, then abnormal curves in the lift of $\mathcal{Z}$ are constant curves.

We observe that on $\lift M$, along the lift $\lift \gamma$ the step of the lifted sub-Riemannian structure  is equal to $3$ for $t\neq \frac{T}{2}$ and is equal to $4$ for $t= \frac{T}{2}$ (due to $\mathrm{rank}(d\beta_{1},d\beta_{2})=2$ at $\mathcal{Z}$, cf.\ Theorem~\ref{intro:step-rank}).

\end{example}

\begin{example}
\label{ex:crossing1}

The last example (cf.\ Example~\ref{ex:rank1}) is a horizontal magnetic field for which the zero set $\mathcal{Z}$ is a horizontal curve and there exists a characteristic curve of the magnetic field which enters into $\mathcal{Z}$ transversally.
Let us consider in $\mathbb{H}$ 
\begin{equation}
\B= 4z\,dx\wedge\omega+x^2\,dy\wedge\omega.
\end{equation}
We have that $\mathcal{Z}=\{x=z=0\}$ is  the $y$-axis, which is a horizontal curve. 
In the complement of $\mathcal{Z}$, characteristic curves for $\beta$ are tangent to the vector field
\begin{equation} \label{eq:dbvf}
-x^{2}X_1+4zX_2=-x^{2}\p_x+4z\p_y+\frac{x}{2}(xy+4z)\partial_{z}.
\end{equation}
Consider the same curve $\gamma:[0,T]\to \mathbb{H}$ as in the previous example
\begin{equation}
\gamma (t) = \begin{cases}
\gamma_{-}(t)=(\frac T 2 - t, 0,0), & t\in\left[0,\frac{T}{2}\right]\\
\gamma_{+}(t)=(0,t- \frac T 2,0),  & t\in\left(\frac{T}{2}, T\right]
\end{cases}
\end{equation}

We have that $\gamma_{-}(t)$ is a reparametrization of an integral curve of \eqref{eq:dbvf}, while $\gamma_{+}(t)$ is a horizontal curve contained in $\mathcal{Z}=\{x=z=0\}$.

Finally one can compute the step of the lifted distribution $\lift\distr$ at points of $\lift \gamma$ thanks to Theorem~\ref{intro:step-rank}. On the set $\mathcal{Z}$ one easily check that $d\beta_{2}=0$ while $d\beta_{1}\neq0$. The surface $\Sigma=\beta_{1}^{-1}(0)=\{z=0\}$ has a unique characteristic point at the origin.  

Hence one can conclude that:
\begin{itemize}
\item[(i)]  the step is equal to $3$ at $\lift\gamma(t)$ for $t\in [0,\frac{T}{2})$ since we are not in $\mathcal{Z}$, 
\item[(ii)] the step is equal to $5$ at $\lift\gamma(t)$ for $t=\frac{T}{2}$ since $\gamma(\frac{T}{2})$ is a characteristic point for $\Sigma$, 
\item[(iii)] the step is equal to $4$ at $\lift\gamma(t)$ for $t\in (\frac{T}{2},T]$ since $\gamma(t)$ is not a characteristic point for $\Sigma$.
\end{itemize}

\end{example}

\bibliography{Mag_SR.bib}

\end{document}